\documentclass[12pt,a4paper]{article}

\usepackage{amsfonts,amsmath, amssymb}
\usepackage{amsthm}
\usepackage{graphicx}
\usepackage{longtable}
\usepackage{float}
\usepackage{array}
\usepackage{enumerate}
\usepackage{threeparttable}
\usepackage{mathrsfs}
\usepackage[text={15cm,24cm},top=2cm,margin=2cm]{geometry} 
\usepackage{listings}
\usepackage{url}
\usepackage{subfigure}
\usepackage{multirow}
\usepackage{booktabs}
\usepackage{textcomp,booktabs}
\usepackage[usenames,dvipsnames]{color}
\usepackage{colortbl}
\usepackage{lscape}
\usepackage{bm}
\usepackage[numbers,square]{natbib}
\usepackage[colorlinks=true,urlcolor=black,citecolor=blue,linkcolor=black,bookmarks=true]{hyperref} 
\usepackage{verbatim}
\usepackage{ulem}
\usepackage{rotating}

\definecolor{mygray}{gray}{.9}
\usepackage[labelfont={bf,it},textfont={bf,it}]{caption}
\DeclareGraphicsExtensions{.jpg,.pdf,.gif,.png}

\newtheorem{proposition}{Proposition}
\newtheorem{theorem}{Theorem}

\newtheorem{lemma}{Lemma}

\newcommand{\FS}[2]{\displaystyle\frac{#1}{#2}}

\newcommand{\R}{\mathbb{R}}

\numberwithin{equation}{section}
\numberwithin{definition}{section}
\numberwithin{remark}{section}
\numberwithin{theorem}{section}
\numberwithin{proposition}{section}
\numberwithin{lemma}{section}
\numberwithin{remark}{section}
\numberwithin{example}{section}
\numberwithin{figure}{section}
\numberwithin{conjecture}{section}
\numberwithin{table}{section}

\DeclareMathOperator\tr{tr}

\DeclareMathOperator\Span{Span}
\DeclareMathOperator\erf{erf}

\begin{document}

\title{Point Forces in Elasticity Equation and Their Alternatives in Multi Dimensions}

\author{Q. Peng, F.J. Vermolen}
\date{July 29, 2020}
\maketitle

\begin{abstract}
We consider several mathematical issues regarding models that simulate traction forces exerted by cells. Since the size of cells is much smaller than the size of the domain of computation, one often considers point forces, modelled by Dirac Delta distributions on boundary segments of cells. In the current paper, we treat the forces that are directed normal to the cell boundary and that are directed toward the cell centre. Since it can be shown that there exists no smooth solution, at least not in $\boldsymbol{H^1}$ for solutions to the governing momentum balance equation, we analyse the convergence and quality of approximation. Furthermore, the expected finite element problems that we get necessitate to scrutinize alternative model formulations, such as the use of smoothed Dirac distributions, or the so-called smoothed particle approach as well as the so-called 'hole' approach where cellular forces are modelled through the use of (natural) boundary conditions. In this paper, we investigate and attempt to quantify the conditions for consistency between the various approaches. This has resulted into error analyses in the $L^2$-norm of the numerical solution based on Galerkin principles that entail Lagrangian basis functions. The paper also addresses well-posedness in terms of existence and uniqueness. The current analysis has been performed for the linear steady-state (hence neglecting inertia and damping) momentum equations under the assumption of Hooke's law. 
	
\end{abstract}

\section{Introduction}
\label{Intro}
Wound healing is a complicated process of a sequence of cellular events contributing to resurfacing, reconstitution and restoration of the tensile strength of injured skin. Significant damage of dermal tissue often leads to skin contraction. If the contraction of the skin near a joint is large then it may result into a decrease of functionality. If the patient's daily life is impacted as result of the contraction, then one speaks of a contracture. 

In order to improve the patient's quality of life, one aims at reducing the contractile behavior of the skin. To reduce the severity of the contraction, one needs to know the physiological dynamics and time evolution of the underlying biological mechanisms. According to \cite{eichler2006modeling,haertel2014transcriptional,li2011fibroblasts}, the contraction starts developing during the proliferative phase of wound healing. This proliferative phase sets in after the inflammatory phase, in which the immune system is clearing up the debris that resulted from the damage. The proliferative phase usually starts from the second day post-wounding, and commonly lasts two to four weeks. Besides the closure of the epidermis (that is the top layer of skin), the proliferative phase is characterized by ingress of fibroblasts from the surrounding undamaged tissue and differentiation to myofibroblasts, and by the regeneration of collagen by the (myo)fibroblasts. Despite the relatively quick closure of the epidermis, often the restoration of the underlying dermis is still in progress. After closure of the epidermis, the damaged region in the dermis is referred to as a scar instead of a wound. Next to the regeneration of collagen, the (myo)fibroblasts exert contractile forces on their direct surroundings, which will result into contraction of the scar tissue. In human skin, typically volume reductions of 5 - 10\% are commonly observed \citep{enoch2008basic}. 

The current manuscript contains an extension of the work in \cite{koppenol2017biomedical}, which treats a model for the contractile forces exerted by the (myo)fibroblasts. The forces are distinguished into two categories: (1) temporary forces that are exerted as long as the (myo)fibroblasts are actively pulling; and (2) permanent or plastic forces, which are imaginary forces that are introduced to describe the localized plastic deformations of the tissue. This formalism was firstly developed by \citet{vermolen2015semi}, and later extended by \citet{boon2016multi}. The formalism is based on the point forces, which are mathematically incorporated by means of linear combinations of Dirac Delta distributions. The irregular nature of Dirac Delta distributions make the solution to the elliptic boundary value problem from the balance of momentum have a singular solution in the sense that for dimensionality higher than one, no formal solutions in the finite-element space $\boldsymbol{H^1}$ exist. Although in classical finite-element strategies, one uses for instance piecewise linear Lagrangian elements, of which the basis functions are in $\boldsymbol{H^1}$, and therewith one attempts to approximate the solution (which is not in $\boldsymbol{H^1}$) as well as possibly by a function in $\boldsymbol{H^1}$. \citet{bertoluzza2018local} demonstrated the convergence of finite-element solutions by means of piecewise linear Lagrangian elements in multiple dimensions. In our earlier studies \citep{peng2019numerical, peng2019pointforces}, we proved the convergence of solutions obtained by regularization of Dirac Delta distributions, the so-called smoothed particle approach and the so-called 'hole' approach to the solution obtained by Dirac Delta distributions in the one- and two-dimensional cases. In the one-dimensional case, for the sake of completeness, we start with the presentation of force equilibrium with point forces, the equations are given by
\begin{eqnarray}
-\frac{d\sigma}{dx}&=f,&\quad\mbox{Equation of Equilibrium,}\\
\epsilon&=\frac{du}{dx},&\quad\mbox{Strain-Displacement Relation,}\\
\sigma&=E\epsilon,&\quad\mbox{Constitutive Equation.}
\end{eqnarray}
To simplify the equation, we use $E=1$ here, the equations above can be combined to the one-dimensional Laplace equation:
\begin{equation}
-\frac{d^2u}{dx^2}=f.
\end{equation}
We assume that there is a biological cell with size $h$ and centre position $c$ in the computational domain such that 
$0 < c -h/2 < c < c + h/2 < L$. Then the force is given by $f=\delta(x-(c-h/2))-\delta(x-(x+h/2))$. Combined with homogeneous Dirichlet boundary conditions: $$u(0)=0, 
~u(L)=0,$$ the Galerkin form is given by
\[ \left\{
\begin{aligned}
&\text{Find $u_h\in H^1_0((0,L))$, such that} \int_{\Omega}\nabla u_h\nabla\phi_hd\Omega=\phi_h(c-\FS{h}{2})-\phi_h(c+\FS{h}{2}),\\
&\text{for all $\phi_h\in H^1_0((0,L))$.}
\end{aligned}	
\right.\]
The exact solution is $$u(x)=\frac{hx}{L} + (x-(c+\FS{h}{2}))_+ - (x-(c-\FS{h}{2}))_+,$$ where $(x)_+=\max\{0,x\}$.
Note that in one dimension, the solution is piecewise linear and hence in $H^1(\Omega)$, however not in $H^2(\Omega)$. Since most
conventional errors are expressed in the $L^2$--norm of the
second derivative of the solution, one may not apriorily expect
very accurate finite element solutions. 

In the current manuscript we extend the results to general dimensionality. The boundary value problem is stated in Section \ref{Elas_multi}. The 'hole' approach and the smoothed particle approach are developed in Section \ref{Alternative}. Furthermore, we prove consistency between all the alternatives and the immersed boundary approach in multi dimensions. Section \ref{conclusion} displays some conclusions and discussions.


\section{Elasticity Equation with Point Sources in Multi Dimensions}
\label{Elas_multi}
Let $\Omega$ be a bounded domain in $\mathbb{R}^n$, then we consider the following balance of momentum where inertial effects have been neglected:
\begin{equation}
\label{Eq_momentum}
-\nabla\cdot\boldsymbol{\sigma}=\boldsymbol{f}.
\end{equation}
Here $\boldsymbol{\sigma}$ denotes the stress tensor and $\boldsymbol{f}$ represents a body force that is exerted within $\Omega$. We consider a linear, homogeneous, isotropic and continuous material; hence, Hooke's Law is used here for the relation between the stress and strain tensors:
\begin{equation}
\label{Eq_sigma}
\boldsymbol{\sigma}=\frac{E}{1+\nu}\left\lbrace \boldsymbol{\epsilon}+\tr(\boldsymbol{\epsilon})\left[\frac{\nu}{1-2\nu} \right]\boldsymbol{I}\right\rbrace,  
\end{equation}
where $E$ is the stiffness of the computational domain, $\nu$ is Poisson's ratio and $\boldsymbol{\epsilon}$ is the infinitesimal Eulerian strain tensor:
\begin{equation}
\label{Eq_epsilon}
\boldsymbol{\epsilon}=\frac{1}{2}\left[ \boldsymbol{\nabla u}+(\boldsymbol{\nabla u})^T\right].
\end{equation}

Within the domain of computation, $\Omega$, we consider the presence of a biological cell, which occupies the portion $\Omega_C$ that is completely embedded within $\Omega$ (hence $\Omega_C $ is a strict subset of $\Omega$). The boundary of the cell $\Gamma_C$ is divided into surface elements. On the centre of each surface element, a point force by means of Dirac Delta distributions, is exerted in the direction of the normal vector that is directed inward into the cell. This results into (see \citep{vermolen2015semi}):
\begin{equation}
\label{Eq_ElasTempeq}
\boldsymbol{f}_t=\sum_{j=1}^{N_S}P(\boldsymbol{x}_j,t)\boldsymbol{n}(\boldsymbol{x}_j)\delta(\boldsymbol{x}-\boldsymbol{x}_j(t))\Delta S(\boldsymbol{x}_j(t)),
\end{equation}  
where $N_S$ is the number of surface elements of the cell, $P(\boldsymbol{x},t)$ is the magnitude of the pulling force exerted at point $\boldsymbol{x}$ and time $t$ per unit of measure (being area in $\mathbb{R}^3$ or length in $\mathbb{R}^2$), $\boldsymbol{n}(\boldsymbol{x})$ is the unit inward pointing normal vector (towards the cell centre) at position $\boldsymbol{x}$, $\boldsymbol{x}_j(t)$ is the midpoint on surface element $j$ of the cell at time $t$ and $\Delta S(\boldsymbol{x}_j)$ is the measure of the surface element $j$. In the general model where we use this principle, we consider transient effects due to migration and possible deformation of the cells. However, since we predominantly focus on the mathematical issues in the current manuscript, we will not consider any time-dependencies and hence $t$ will be removed from the expressions in the remainder of the paper.

In the n-dimensional case, we are solving the boundary value problems described in Eq (\ref{Eq_momentum}), (\ref{Eq_sigma}) and (\ref{Eq_epsilon}). The body force is given in Eq (\ref{Eq_ElasTempeq}). Therefore, the immersed boundary value problem that we are going to consider is given by
\begin{equation*}
\label{Eq_direct}
(BVP)\left\{
\begin{aligned}
-\nabla\cdot\boldsymbol{\sigma}(\boldsymbol{x})&=\sum_{j=1}^{N_S}P(\boldsymbol{x}_j)\boldsymbol{n}(\boldsymbol{x}_j)\delta(\boldsymbol{x}-\boldsymbol{x}_j)\Delta S(\boldsymbol{x}_j),&\mbox{in $\Omega$,}\\
\boldsymbol{u}&=\boldsymbol{0},&\mbox{on $\partial\Omega$.}
\end{aligned}
\right.
\end{equation*} 
Next to this boundary value problem, we consider the continuous immersed boundary 
counterpart, given by 
\begin{equation*}
\label{Eq_direct2}
(BVP_{\infty})\left\{
\begin{aligned}
-\nabla\cdot\boldsymbol{\sigma}(\boldsymbol{x})&=\int_{\Gamma_C}P(\boldsymbol{x}')\boldsymbol{n}(\boldsymbol{x}')\delta(\boldsymbol{x}-\boldsymbol{x}')dS(\boldsymbol{x}'),&\mbox{in $\Omega$,}\\
\boldsymbol{u}&=\boldsymbol{0},&\mbox{on $\partial\Omega$,}
\end{aligned}
\right.
\end{equation*} 
where we take $N_s\rightarrow\infty$. Thus, the body force reads as 
\begin{equation}
\label{Eq_force_int}
\boldsymbol{f}_t^\infty = \int_{\Gamma_C}P(\boldsymbol{x}')\boldsymbol{n}(\boldsymbol{x}')\delta(\boldsymbol{x}-\boldsymbol{x}')dS(\boldsymbol{x}').
\end{equation}

Due to the irregular nature of the Dirac Delta distributions, the solutions do not exist in $\boldsymbol{H^1}$. We attempt to approximate the solution by the functions in $\boldsymbol{H^1}$ via the Galerkin form of $(BVP)$ and $(BVP_\infty)$. In this manuscript, piecewise linear Lagrangian basis functions are selected. Further, the convergence of the finite-element solutions using linear Lagrangian elements in general dimensionality has been proved in \cite{bertoluzza2018local}.

To construct the Galerkin form, we introduce the bilinear form
$a(.,.)$
\begin{equation}
\label{Eq_a}
a(\boldsymbol{u}_h,\boldsymbol{\phi}_h)=\int_{\Omega}\boldsymbol{\sigma}(\boldsymbol{u}_h):\nabla\boldsymbol{\phi}_hd\Omega =
\int_{\Omega}\boldsymbol{\sigma}(\boldsymbol{u}_h):\boldsymbol{\epsilon}(\boldsymbol{\phi}_h)d\Omega,
\end{equation}
where the last step is motivated by symmetry of the stress tensor
$\boldsymbol{\sigma}$.
Since the solution $\boldsymbol{u}$ is not in $\boldsymbol{H^1}(\Omega)$, we consider a subspace of $\boldsymbol{H^1}(\Omega)$, which is defined as $\boldsymbol{V_h}(\Omega)=\Span\{\boldsymbol{\phi^1},\boldsymbol{\phi^2}, \dots, \boldsymbol{\phi^N}\}$ \citep{scott1973finite}. Here, $\boldsymbol{\phi^i}$ for $i=\{1,2,\dots,N\}$ is the linear Lagrangian basis function, which is piecewise smooth and continuous over $\Omega$. Hence, these basis functions are in $\boldsymbol{H^1}$. Subsequently, the Galerkin form is 
\[(GF)\left\{
\begin{aligned}
&\text{Find $\boldsymbol{u}_h\in\boldsymbol{V}_h(\Omega)$, such that }\\ &a(\boldsymbol{u}_h,\boldsymbol{\phi}_h)=(\boldsymbol{\phi}_h,\boldsymbol{f}_t)=\sum_{j=1}^{N_S}P(\boldsymbol{x}_j)\boldsymbol{n}(\boldsymbol{x}_j)\boldsymbol{\phi}_h(\boldsymbol{x}_j)\Delta S(\boldsymbol{x}_j), \\
&\text{for all $\boldsymbol{\phi}_h\in \{\boldsymbol{\phi^1},\boldsymbol{\phi^2}, \dots, \boldsymbol{\phi^N}\}\subset\boldsymbol{V}_h(\Omega)$.}
\end{aligned}
\right.\]
We further consider the solution to the continuous immerse boundary problem, with the
following Galerkin form:
\[(GF_{\infty})\left\{
\begin{aligned}
&\text{Find $\boldsymbol{u}_h\in\boldsymbol{V}_h(\Omega)$, such that }\\ 
&a(\boldsymbol{u}_h,\boldsymbol{\phi}_h)=(\boldsymbol{\phi}_h,\boldsymbol{f}^\infty_t)=\int_{\Gamma_C} P(\boldsymbol{x}')\boldsymbol{n}(\boldsymbol{x}')\boldsymbol{\phi}_h(\boldsymbol{x}')
dS(\boldsymbol{x}'), \\
&\text{for all $\boldsymbol{\phi}_h\in \{\boldsymbol{\phi^1},\boldsymbol{\phi^2}, \dots, \boldsymbol{\phi^N}\}\subset\boldsymbol{V}_h(\Omega)$.}
\end{aligned}
\right.\]

Before we proceed to claim the existence and the uniqueness of the Galerkin solution in $(GF)$, we state Korn's Inequality in multiple dimensions:
\begin{lemma}\label{lemma_korn}{\bf (Korn's Second Inequality\citep{braess2007finite})}
Let $\Omega\subset\R^n$ be an open, bounded and connected domain.
Then there exists a positive constant $K$, such that for any vector-valued function $\boldsymbol{u}\in \boldsymbol{H^1_0}(\Omega)$,
$$\int_{\Omega}||\boldsymbol{\epsilon} (\boldsymbol{u})||^2d\Omega
\geqslant K\|\boldsymbol{u}\|^2_{\boldsymbol{H^1}(\Omega)}.$$ 	
\end{lemma}
We note that Korn's Second Inequality can be generalised to cases in
which the boundary condition $\boldsymbol{u} = \boldsymbol{0}$ is
imposed only on a non-zero measure part of the boundary. Using Korn's Second
inequality gives the following lemma:
\begin{lemma}\label{lemma_coerc}
Let $\Omega\subset\R^n$ be an open, bounded and connected domain.
Then there exists a positive constant $K$, such that for any vector-valued function $\boldsymbol{u}\in \boldsymbol{H^1_0}(\Omega)$,
$$a(\boldsymbol{u},\boldsymbol{u}) = \int_{\Omega} \boldsymbol{\sigma}(\boldsymbol{u}) : 
\boldsymbol{\epsilon}(\boldsymbol{u}) d \Omega \geqslant
K ||\boldsymbol{u}||^2_{\boldsymbol{H^1}(\Omega)}.$$
\end{lemma}
\begin{proof}
The lemma directly follows from the definition of the stress tensor,
let $\boldsymbol{u} \in \boldsymbol{H^1_0}(\Omega)$:

\begin{align*}
a(\boldsymbol{u},\boldsymbol{u}) &= \int_{\Omega}\boldsymbol{\sigma}(\boldsymbol{u}):\boldsymbol{\epsilon}(\boldsymbol{u})d\Omega =
\int_{\Omega} \frac{E}{1+\nu}\left\{ \boldsymbol{\epsilon}(\boldsymbol{u})+\text{tr}(\boldsymbol{\epsilon}(\boldsymbol{u}))
\frac{\nu}{1-2 \nu}
\boldsymbol{I}\right\} :\boldsymbol{\epsilon}(\boldsymbol{u}) d \Omega \\ 
&= \int_{\Omega} \frac{E}{1+\nu} ||\boldsymbol{\epsilon}(\boldsymbol{u})||^2 + \frac{E \nu}{(1+\nu)(1-2\nu)}
(\tr(\boldsymbol{\epsilon}(\boldsymbol{u})))^2 d \Omega\\ 
&\geqslant \frac{E}{1+\nu} K
|| \boldsymbol{u} ||^2_{\boldsymbol{H^1}(\Omega)}.
\end{align*}

The last step follows from Lemma \ref{lemma_korn}. Hence, redefining
$K := \frac{E}{1+\nu} K$ concludes
the proof the lemma.
\end{proof}

Herewith, coerciveness of the linear form $a(.,.)$ has been demonstrated, which
is needed for the proof of existence and uniqueness of the 
Galerkin finite-element solution.
\begin{theorem}\label{Th_solexist}
Let $\{\boldsymbol{\phi^i}\}$ be piecewise Lagrangian basis field functions and let $\boldsymbol{F}$ be a vector
in $\mathbb{R}^n$ with unit length, further let $P\in C(\overline{\Omega})$,
and let $|P| \leqslant M_2$ for some $M_2 > 0$. We define $\boldsymbol{V}_h(\Omega)=\Span\{\boldsymbol{\phi^1},\boldsymbol{\phi^2}, \dots, \boldsymbol{\phi^N}\}\subset \boldsymbol{H^1_0}(\Omega)$, then
\begin{itemize}
	\item $\exists ~ ! ~ \boldsymbol{u}^G_h(\boldsymbol{x};\boldsymbol{x}';\boldsymbol{F})\in\boldsymbol{V}_h(\Omega)$ such that
	$a(\boldsymbol{u}_h,\boldsymbol{\phi}_h)=
	\boldsymbol{F}(\boldsymbol{x}') \cdot \boldsymbol{\phi}_h(\boldsymbol{x}') $
	for all $\boldsymbol{\phi}_h\in\boldsymbol{V}_h$; 
	\item $\exists ~ ! ~ \boldsymbol{u}_h\in\boldsymbol{V}_h(\Omega)$ such that
	$a(\boldsymbol{u}_h,\boldsymbol{\phi}_h)=\sum_{j=1}^{N_S}P(\boldsymbol{x}_j)\boldsymbol{n}(\boldsymbol{x}_j)\boldsymbol{\phi}_h(\boldsymbol{x}_j)\Delta S(\boldsymbol{x}_j)$ for all $\boldsymbol{\phi}_h\in\boldsymbol{V}_h$, and 
	$\boldsymbol{u}_h = \sum_{j=1}^{N_S} P({\bf x}_j) 
	{\bf u}_h^G({\bf x};{\bf x}_j;{\bf n}({\bf x}_j)) \Delta 
	S({\bf x}_j) $;  
	\item $\exists ~ ! ~ \boldsymbol{u}_h\in\boldsymbol{V}_h(\Omega)$ such that
	$a(\boldsymbol{u}_h,\boldsymbol{\phi}_h)=\int_{\Gamma_C}P(\boldsymbol{x'})\boldsymbol{n}(\boldsymbol{x'})\boldsymbol{\phi}_h(\boldsymbol{x'})dS(\boldsymbol{x}')$ for all $\boldsymbol{\phi}_h\in\boldsymbol{V}_h$, and
	$\boldsymbol{u}_h=\int_{\Gamma_C} P({\bf x}') {\bf u}_h^G({\bf x};{\bf x}';{\bf n}({\bf x}')) dS(\boldsymbol{ x}')$;
\end{itemize}
\end{theorem}
\begin{proof}
\begin{itemize}
	\item It is immediately clear that $a(.,.)$ is a bilinear form. We have 
	$\boldsymbol{V}_h \subset \boldsymbol{H^1_0}(\Omega)$,
	and $a(.,.)$ is bounded in $\boldsymbol{H_0^1}(\Omega)$ (see for instance
	\citep{atkinson2005theoretical}). Furthermore, Lemma \ref{lemma_coerc} says that $a(.,.)$ is
	coercive in $\boldsymbol{H}_0^1(\Omega)$. Regarding the right-hand side, we have
	$|\boldsymbol{\phi}_h| \le M_1$ for some $M_1 > 0$ since $\boldsymbol{\phi}_h$ is a
	Lagrangian function, and hence the magnitude of the 
	right-hand side can be bounded from above by
	$$
	|\boldsymbol{F} \cdot \boldsymbol{\phi}_h(\boldsymbol{x}')| \leqslant M_1.
	$$
	Note that $||\boldsymbol{F}|| = 1$. Hence the right-hand side is bounded, since we are looking for a solution in a finite dimensional space $\boldsymbol{V}_h$, the system $$A \boldsymbol{c} = \boldsymbol{b},$$ where the coefficients of the symmetrix matrix $A$ are defined by $a_{ij} = a(\phi_i,\phi_j)$, and where a limited number of entries of $\boldsymbol{b}$ are non-zero and given by $\boldsymbol{F} \cdot \boldsymbol{\phi}_h(\boldsymbol{x}')$, which is finite. Since $\boldsymbol{b}$ is finite, and  $A$ is invertible, existence and uniqueness of $\boldsymbol{u}_h$ follow (one could apply Lax-Milgram's theorem on the space $\mathbb{R}^n$ in this context) from the algebraic system.
	\item Existence and uniqueness follow analogously, only boundedness of the right-hand
	side, which is a linear functional in $\boldsymbol{\phi}_h \in \boldsymbol{V}_h(\Omega)$ 
	has to be checked: 
	\begin{align*}
	&|\sum_{j=1}^{N_S} P(\boldsymbol{x}_j)\boldsymbol{n}(\boldsymbol{x}_j) \cdot
	\boldsymbol{\phi}_h(\boldsymbol{x}_j) \Delta S(\boldsymbol{x}_j)| \\
	&\leqslant
	\sum_{j=1}^{N_S} |P(\boldsymbol{x}_j)| ||\boldsymbol{n}(\boldsymbol{x}_j)|| 
	||\boldsymbol{\phi}_h(\boldsymbol{x}_j)|| \Delta S(\boldsymbol{x}_j)  \\
	&=\sum_{j=1}^{N_S} |P(\boldsymbol{x}_j)| 
	||\boldsymbol{\phi}_h(\boldsymbol{x}_j)|| \Delta S(\boldsymbol{x}_j) \leqslant
	M_1 M_2 \sum_{j=1}^{N_S} \Delta S(\boldsymbol{x}_j).
	\end{align*}
	
	Note that $\boldsymbol{n}$ has unit length.
	The summation gives the polygonal length or polyhedral area of the cell boundary.
	Hence the right-hand side is bounded, then by Lax-Milgram's Lemma, existence 
	and uniqueness follow. Further by substitution, it follows that
	that
	\begin{align*}
	a(\boldsymbol{u}_h,\boldsymbol{\phi}_h) &= a(\sum_{j=1}^{N_S} P(\boldsymbol{x}_j)
	\boldsymbol{u}^G_h(\boldsymbol{x},\boldsymbol{x}_j,\boldsymbol{n}(\boldsymbol{x}_j))
	\Delta S(\boldsymbol{x}_j),\boldsymbol{\phi}_h)  \\
	&=\sum_{j=1}^{N_S} P(\boldsymbol{x}_j) 
	a(\boldsymbol{u}^G_h(\boldsymbol{x},\boldsymbol{x}_j,\boldsymbol{n}(\boldsymbol{x}_j)),
	\boldsymbol{\phi}_h) \Delta S(\boldsymbol{x}_j) \\
	&=\sum_{j=1}^{N_S}P(\boldsymbol{x}_j)\boldsymbol{n}(\boldsymbol{x}_j) \cdot
	\boldsymbol{\phi}_h(\boldsymbol{x}_j)\Delta S(\boldsymbol{x}_j).
	\end{align*}
	The last step uses the first part of the theorem, and finally the assertion is proved similarly to the first assertion.
	\item We proceed similarly, by boundedness of the right-hand side:
	$$
	|\int_{\Gamma_C} P(\boldsymbol{x}') \boldsymbol{n}(\boldsymbol{x}') \cdot 
	\boldsymbol{\phi_h}(\boldsymbol{x}') d S(\boldsymbol{x}')| \leqslant
	M_1 M_2 |\Gamma_C|,
	$$
	where $|\Gamma_C|$ is the measure of the boundary surface of the biological cell. It again shows that the right-hand side is a bounded linear functional in 
	$\boldsymbol{V}_h(\Omega)$. We proceed by substitution:
	
	\begin{align*}
	a(\boldsymbol{u}_h, \boldsymbol{\phi}_h) &= a(\int_{\Gamma_C} P(\boldsymbol{x}')
	\boldsymbol{u}^G_h(\boldsymbol{x},\boldsymbol{x}',\boldsymbol{n}(\boldsymbol{x}'))
	d S(\boldsymbol{x}'),\boldsymbol{\phi}_h) \\
	&=\int_{\Gamma_C} P(\boldsymbol{x}') 
	a(\boldsymbol{u}^G_h(\boldsymbol{x},\boldsymbol{x}',\boldsymbol{n}(\boldsymbol{x}')),
	\boldsymbol{\phi}_h) d S(\boldsymbol{x}')\\
	&=\int_{\Gamma_C} P(\boldsymbol{x}')\boldsymbol{n}(\boldsymbol{x}') \cdot 
	\boldsymbol{\phi}_h(\boldsymbol{x}') d S(\boldsymbol{x}').
	\end{align*}
	
\end{itemize}
Note that, formally, it was not necessary to prove boundedness, since coerciveness implies
uniqueness and the existence was proved by construction and by combining the
result for the existence of $\boldsymbol{u}_h^G$.
\end{proof}

Note that for the 'continuous' weak formulation, there is no solution in $\boldsymbol{H^1}$, hence the above claim demonstrates the existence and uniqueness of a Galerkin-based approximation in a subset of $\boldsymbol{H^1}$ to a function that is not in $\boldsymbol{H^1}$. The situation is somewhat comparable to approximating $\sqrt{2} \notin \mathbb{Q}$ arbitrarily accurately by a sequence of successive approximations in $\mathbb{Q}$.
Further in two- and three- dimensional case, the convergence between the solution to $(GF)$ and $(GF_\infty)$ can be proved. Similar work has been done in \cite{lacouture2015numerical} regarding Stokes problem with the Delta distribution term. 
\begin{theorem}
Let $\Gamma_C$ be a polygon or polyhedron embedded in $\Omega\subset\R^n$ and let $P(\boldsymbol{x})$ be sufficiently smooth. Further,
let $\boldsymbol{x}_j$ be the midpoint of surface element $\Delta S(\boldsymbol{x}_j)$.
Denote $\boldsymbol{u}_h^{\Delta S}$ as the Galerkin solution to $(GF)$ and the $\boldsymbol{u}_h^\infty$ as the  Galerkin solution to $(GF_{\infty})$, respectively. In two dimensions, for any $\boldsymbol{x}\notin\Gamma_C$, there exists a positive constant $C$, such that for each component of $\boldsymbol{u}_h^{\infty}$ we have $$|\boldsymbol{u}_h^{\Delta S}-\boldsymbol{u}_h^\infty|\leqslant C\Delta S_{max}^2,$$ where $\Delta S_{max}=\max\{\Delta S(\boldsymbol{x}_j)\}$ for any $j=\{1,2,\cdots, N_S\}$. In three dimensions, for any $\boldsymbol{x}\notin\Gamma_C$, there exists a positive constant $C$, such that for each component of $\boldsymbol{u}_h^{\infty}$ we have $$|\boldsymbol{u}_h^{\Delta S}-\boldsymbol{u}_h^\infty|\leqslant C h_{max}^2,$$ where $h_{max}$ is the maximal diameter among all the triangular elements over $\Gamma_C$.  	
\end{theorem}
\begin{proof}
Away from $\Gamma_C$, the function $\boldsymbol{u}_h^G$ is smooth, and since $P(\boldsymbol{x})$ is 
smooth as well, the integrand, given by $P(\boldsymbol{x})\boldsymbol{u}_h^G$ is smooth as well. 
For ease of notation, we set 
$\boldsymbol{f}(\boldsymbol{x})=P(\boldsymbol{x})\boldsymbol{u}_h^G(\boldsymbol{x};\boldsymbol{x'};\boldsymbol{n})$. We start with
the 2D-case.
Given the $i-th$ boundary element $\Delta S_i$ on $\Gamma_C$ with the endpoints $\boldsymbol{x}_i$ and $\boldsymbol{x}_{i+1}$ and we denote its midpoint by $\boldsymbol{x}_{i+1/2}$, where $i\in\{1,2,\cdots, N_S\}$. We consider $$\boldsymbol{x}(s)=\boldsymbol{x}_{i+1/2}+s\FS{\boldsymbol{x}_{i+1}-\boldsymbol{x}_i}{2},\quad-1\leqslant s\leqslant 1,$$
Hence, $\boldsymbol{x}(0)=\boldsymbol{x}_{i+1/2}$ and $\boldsymbol{x}'(s)=\FS{1}{2}(\boldsymbol{x}_{i+1}-\boldsymbol{x}_1),$
and subsequently $$\|\boldsymbol{x}'(s)\|=\FS{1}{2}\|\boldsymbol{x}_{i+1}-\boldsymbol{x}_1\|.$$  	
We calculate the contribution over $\Delta S_i$ to the integral, where Taylor's
Theorem and the Mean Value Theorem for integration are used to warrant the existence 
of a $\hat{s} \in (-1,1)$, such that
\begin{align*}
&\int_{\Delta S_i}\boldsymbol{f}(\boldsymbol{x})dS=\int_{-1}^{1}\boldsymbol{f}(\boldsymbol{x}(s))\|\boldsymbol{x}'(s)\|ds\\
&=\FS{1}{2}\|\boldsymbol{x}_{i+1}-\boldsymbol{x}_i\|\int_{-1}^{1}\boldsymbol{f}(\boldsymbol{x}(s))ds\\
&\text{(Taylor Expansion)}  =\FS{1}{2}\|\boldsymbol{x}_{i+1}-\boldsymbol{x}_i\|\int_{-1}^{1} \boldsymbol{f}(\boldsymbol{x}(0))+s\FS{\boldsymbol{x}_{i+1}-\boldsymbol{x}_i}{2}\nabla \boldsymbol{f}(\boldsymbol{x}(s))|_{s=0}\\
&+\FS{1}{2}s^2(\FS{\boldsymbol{x}_{i+1}-\boldsymbol{x}_i}{2})^T \boldsymbol{H}(\boldsymbol{x}(\hat{s})) (\FS{\boldsymbol{x}_{i+1}-\boldsymbol{x}_i}{2})ds\\
&=\FS{1}{2}\|\boldsymbol{x}_{i+1}-\boldsymbol{x}_i\|[2\boldsymbol{f}(\boldsymbol{x}_{i+1/2})+0+\FS{1}{12}(\boldsymbol{x}_{i+1}-\boldsymbol{x}_i)^T \boldsymbol{H}(\boldsymbol{x}(\hat{s})) (\boldsymbol{x}_{i+1}-\boldsymbol{x}_i)]\\
&=\|\boldsymbol{x}_{i+1}-\boldsymbol{x}_i\|\boldsymbol{f}(\boldsymbol{x}_{i+1/2})+\FS{1}{24}\|\boldsymbol{x}_{i+1}-\boldsymbol{x}_i\|(\boldsymbol{x}_{i+1}-\boldsymbol{x}_i)^T\boldsymbol{H}(\boldsymbol{x}(\hat{s})) (\boldsymbol{x}_{i+1}-\boldsymbol{x}_i),
\end{align*}
where $\boldsymbol{H}(\boldsymbol{x}(s))$ is the Hessian matrix of $f(\boldsymbol{x}(s))$. Therefore, we obtain that
\begin{align*}
&\left|\int_{\Delta S_i}\boldsymbol{f}(\boldsymbol{x})dS-\|\boldsymbol{x}_{i+1}-\boldsymbol{x}_1\|\boldsymbol{f}(\boldsymbol{x}_{i+1/2})\right|\\
&=\FS{1}{24}\|\boldsymbol{x}_{i+1}-\boldsymbol{x}_i\|\cdot|(\boldsymbol{x}_{i+1}-\boldsymbol{x}_i)^T\boldsymbol{H}(\boldsymbol{x}(\hat{s})) (\boldsymbol{x}_{i+1}-\boldsymbol{x}_i)|\\
&\leqslant \FS{1}{24}\|\boldsymbol{x}_{i+1}-\boldsymbol{x}_i\|\tilde{K}\|\boldsymbol{x}_{i+1}-\boldsymbol{x}_i\|^2.
\end{align*}
Since $\boldsymbol{f}(\boldsymbol{x})\in \boldsymbol{C^2}(\Omega)$, it follows that there exists a $\tilde{K}>0$, such that $$|(\boldsymbol{x},\boldsymbol{H}(\boldsymbol{x}))|\leqslant\tilde{K}\|\boldsymbol{x}\|^2.$$
Therefore, considering the summation of the boundary elements over $\partial\Omega_C$, 
\begin{align*}
&\left|\int_{\Delta S_i}\boldsymbol{f}(\boldsymbol{x})dS-\sum_{i=1}^{N_S}\|\boldsymbol{x}_{i+1}-\boldsymbol{x}_1\|\boldsymbol{f}(\boldsymbol{x}_{i+1/2})\right| \\
&\leqslant\sum_{i=1}^{N_S}\FS{1}{24}\|\boldsymbol{x}_{i+1}-\boldsymbol{x}_i\|\tilde{K}\|\boldsymbol{x}_{i+1}-\boldsymbol{x}_i\|^2\\
&\leqslant \FS{1}{24}\tilde{K}\Delta S_{max}^2\sum_{i=1}^{N_S}\|\boldsymbol{x}_{i+1}-\boldsymbol{x}_i\|\\
& \leqslant \FS{1}{24}\tilde{K}\Delta S_{max}^2|\Gamma_C|,
\end{align*}
where $\Delta S_{max} = \max_{i \in \{1,\ldots,N_S\}}|| \boldsymbol{x}_{i+1} - \boldsymbol{x}_i ||$ is the maximal length of the line segment over $\Gamma_C$, and $|\Gamma_C|$ is the perimeter of the polygon $\Gamma_C$. It can be concluded that there exists a positive constant $K$, such that $$|\boldsymbol{u}_h^\infty-\boldsymbol{u}_h^{\Delta S}|\leqslant K\Delta S_{max}^2.$$

In three dimensions, the surface element is a triangle. We map the triangle in $(x,y,z)$-space to the reference triangle in $(s,t)$-space with points $(0,0),(0,1)$ and $(1,0)$. Suppose there is a surface element $e_j$ with nodal points $\boldsymbol{x}_1,\boldsymbol{x_2}$ and $\boldsymbol{x_3}$, then the centre point of $e_j$ is $\boldsymbol{x}_c=(\boldsymbol{x}_1+\boldsymbol{x}_2+\boldsymbol{x}_3)/3$. The map from the 
reference triangle $e_0$ to the physical triangle $e_j$ is given by $$\boldsymbol{x}(s,t)=\boldsymbol{x}_1(1-s-t)+s\boldsymbol{x}_2+t\boldsymbol{x}_3, \quad0\leqslant s,t\leqslant 1.$$ 
For any function $\boldsymbol{f}(\boldsymbol{x})\in C^2(\Omega)$, the integral over the original triangle is given by $$\int_{e_j}\boldsymbol{f}(\boldsymbol{x})d\boldsymbol{x}=\int_{e_0}\boldsymbol{f}(\boldsymbol{x}(s,t))|\sqrt{\det(\boldsymbol{J}^T\boldsymbol{J})|}d(s,t),$$ where $\boldsymbol{J}$ is the Jacobian matrix, given by
$$
\boldsymbol{J} = \frac{\partial (x,y,z)}{\partial (s,t)} = 
\begin{pmatrix}
x_2 - x_1 & x_3 - x_1 \\
y_2 - y_1 & y_3 - y_1 \\
z_2 - z_1 & z_3 - z_1
\end{pmatrix},
$$
and $\sqrt{|\det(\boldsymbol{J}^T \boldsymbol{J})|}$ is twice the area of the original triangle $e_j$, i.e.
\begin{equation*}
|\Delta_j|:= \sqrt{|\det(\boldsymbol{J}^T \boldsymbol{J})|} = 
||(\boldsymbol{x}_2 - \boldsymbol{x}_1) \times (\boldsymbol{x}_3 - \boldsymbol{x}_1) ||.
\end{equation*} 
We conduct the same process as for the two dimensional case, we obtain, where 
$\boldsymbol{x}(\frac{1}{3},\frac{1}{3}) = \boldsymbol{x}_c$ coincides with the midpoint 
of element $e_j$, and where Taylor's Theorem for 
multi-variate functions is used:
\begin{align*}
&\int_{e_j}\boldsymbol{f}(\boldsymbol{x})d\boldsymbol{x}=\int_{e_0}\boldsymbol{f}(\boldsymbol{x}(s,t))|\Delta_j|d(s,t)\\
&=|\Delta_j|\int_{e_0}\boldsymbol{f}(\boldsymbol{x}(s,t))d(s,t) \\
&=|\Delta_j|\int_{e_0}\boldsymbol{f}(\boldsymbol{x}_c)+(\boldsymbol{x}(s,t)-\boldsymbol{x}_c)\cdot\nabla \boldsymbol{f}(\boldsymbol{x}_c)|\\
&+\frac{1}{2}(\boldsymbol{x}(s,t)-\boldsymbol{x}_c)^T \boldsymbol{H}(\boldsymbol{x}(\hat{s},\hat{t}))(\boldsymbol{x}(s,t)-\boldsymbol{x}_c)d(s,t)\\
&=|\Delta_j|[\frac{1}{2}\boldsymbol{f}(\boldsymbol{x}_c)+0+\frac{1}{2}\int_{e_0}(\boldsymbol{x}(s,t)-\boldsymbol{x}_c)^T \boldsymbol{H}(\boldsymbol{x}(\hat{s},\hat{t}))(\boldsymbol{x}(s,t)-\boldsymbol{x}_c)
d(s,t)].
\end{align*}
Due to $\boldsymbol{f}(\boldsymbol{x})\in \boldsymbol{C^2}(\Omega)$, then for the Hessian matrix of $\boldsymbol{f}(\boldsymbol{x})$, there exists $\tilde{K}>0$, such that $$|(\boldsymbol{x},\boldsymbol{H}(\boldsymbol{x}))|\leqslant\tilde{K}\|\boldsymbol{x}\|^2.$$ It yields
\begin{align*}
&\left|\int_{e_j}\boldsymbol{f}(\boldsymbol{x})d\boldsymbol{x}-\frac{|\Delta_j|}{2}\boldsymbol{f}(\boldsymbol{x}_c)\right|\\
&\leqslant\left|\frac{|\Delta_j|}{2}\int_{e_0}(\boldsymbol{x}(s,t)-\boldsymbol{x}_c)^T \boldsymbol{H}(\boldsymbol{x}(\hat{s},\hat{t}))(\boldsymbol{x}(s,t)-\boldsymbol{x}_c)d(s,t)\right| \\
& \leqslant \frac{|\Delta_j|}{4}\tilde{K} h_{max}^2,
\end{align*}
where $h_{max}^2$ is the largest diameter in the original triangle $e_j$. Considering all the surface elements over $\Gamma_C$, we compute
\begin{align*}
\left|\int_{\Gamma_C}\boldsymbol{f}(\boldsymbol{x})d\boldsymbol{x}-\sum_{j=1}^{N_S}\frac{|\Delta_j|}{2}\boldsymbol{f}(\boldsymbol{x}_j)\right|&\leqslant \frac{\tilde{K}}{4} h_{max}^2\sum_{j=1}^{N_S}\frac{|\Delta_j|}{2}
\leqslant \frac{\tilde{K}}{4} h_{max}^2|\Gamma_C|,
\end{align*}
where $h_{max}^2$ is the maximal diameter among all the surface element (i.e. triangle) and $|\Gamma_C|$ is the sum of the measure (area in $\R^3$) of all the surface elements over $\Gamma_C$. Therefore, in three dimensions, we can conclude that there exists a positive constant $K$, such that for the unique Galerkin solution to both $(GF)$ and $(GF_\infty)$,  
$$|\boldsymbol{u}_h^\infty-\boldsymbol{u}_h^{\Delta S}|\leqslant Kh_{max}^2.$$
\end{proof}

The above proof and theorem can easily be extended to
higher dimensionalities.

\section{Alternative Approaches for Elasticity Equation with Point Sources in Multi Dimensions}
\label{Alternative}
\subsection{The 'Hole' Approach}
\noindent
A different approach is based on considering cellular forces
on the cell boundary by means of a boundary condition. In 
this alternative approach, one 'removes' the cell region
from the domain of computation. Herewith, one creates a 'hole'
in the domain. We consider the balance of momentum
over $\Omega \setminus \overline{\Omega}_C$. This gives
the following boundary value problem:
\begin{equation*}
\label{Eq_elas_hole}
(BVP_H)\left\{
\begin{aligned}
-\nabla\cdot\boldsymbol{\sigma}&=0,\qquad&\mbox{in $\Omega\backslash\overline{\Omega}_C$,}\\
\boldsymbol{\sigma}\cdot\boldsymbol{n}&=P(\boldsymbol{x})\boldsymbol{n}(\boldsymbol{x}),\qquad&\mbox{on $\partial\Omega_C$,}\\
\boldsymbol{u}&=\boldsymbol{0},\qquad&\mbox{on $\partial\Omega$,}
\end{aligned}
\right.
\end{equation*}
where $\boldsymbol{\sigma}$ is defined in Eq (\ref{Eq_sigma}) with stiffness $E$. Let $D \subset \Omega$, then we introduce the following
notation:
$$
a_{D,E}(\boldsymbol{u},\boldsymbol{v}) := \int_D 
\boldsymbol{\sigma}(\boldsymbol{u}) : 
\boldsymbol{\epsilon}(\boldsymbol{v}) d\Omega.
$$
Note that the stiffness can be a constant or a function of space over the domain $D$. 

The corresponding weak form is stated below:
\[
(WF_H) \left \{
\begin{aligned}
&\text{Find $\boldsymbol{u}^{H} \in \boldsymbol{H^1}(\Omega \setminus \Omega_C)$
such that} \\
&a_{\Omega \setminus \Omega_C,E}
(\boldsymbol{u}^{H},\boldsymbol{\phi}) =
\int_{\Gamma_C} P(\boldsymbol{x}) \boldsymbol{n}(\boldsymbol{x}) \cdot \boldsymbol{\phi} d S(\boldsymbol{x}), \text{for all $\boldsymbol{\phi}\in \boldsymbol{H^1}(\Omega \setminus \Omega_C)$.}
\end{aligned}
\right.\]
Since $\boldsymbol{\phi} \in \boldsymbol{H^1}(\Omega \setminus \Omega_C)$, it follows
from the Trace Theorem \citep{braess2007finite}, and by noting that $\boldsymbol{\phi}|_{\partial \Omega} = 0$, that there is a
$C_1 > 0$ such that $||\boldsymbol{\phi}||_{\boldsymbol{L^2}(\Gamma_C)} \le C_1
||\boldsymbol{\phi}||_{\boldsymbol{H^1}(\Omega)}$, which implies that the right-hand
side in the weak form is bounded. Subsequently one 
combines Korn's Inequality with Lax-Milgram's Lemma
to conclude that a unique solution in $\boldsymbol{H^1}$ exists.

We compare the immersed boundary method with the 'hole' approach by taking $\beta \geqslant 0$, then we adjust the immersed boundary method such that
\begin{equation}
\label{Eq_stiffness}
E(\boldsymbol{x}) =
\begin{cases}
\beta E, & \text{in } \Omega_C, \\
E, & \text{\color{red} in  $\Omega \setminus \overline{\Omega}_C$.}
\end{cases}
\end{equation}

Regarding the adjusted immersed boundary approach where the stiffness is given by Eq (\ref{Eq_stiffness}), we
have the following Galerkin form
\[
(GF_{\beta}) \left \{
\begin{aligned}
&\text{Find $\boldsymbol{u}_h^{\beta} \in \boldsymbol{V}_h(\Omega)$
such that for all $\boldsymbol{\phi}_h \in  \boldsymbol{V}_h(\Omega)$, we have} \\
&\beta a_{\Omega_C,E}(\boldsymbol{u}_h^{\beta},
\boldsymbol{\phi}_h) + 
a_{\Omega \setminus \Omega_C,E}
(\boldsymbol{u}_h^{\beta},\boldsymbol{\phi}_h)=\int_{\Gamma_C} P(\boldsymbol{x}) \boldsymbol{n}(\boldsymbol{x}) \cdot \boldsymbol{\phi}_h(\boldsymbol{x}) d S(\boldsymbol{x}),
\end{aligned}
\right.\]
where $\boldsymbol{V}_h(\Omega)$ is defined in Theorem \ref{Th_solexist} in Section \ref{Elas_multi}.

For the 'hole' approach, we have the following Galerkin
form
\[
(GF_H) \left \{
\begin{aligned}
&\text{Find $\boldsymbol{u}_h^{H} \in  \boldsymbol{V}_h(\Omega \setminus \Omega_C)$
such that for all $\boldsymbol{\phi}_h \in  \boldsymbol{V}_h(\Omega \setminus \Omega_C)$
we have} \\ 
&a_{\Omega \setminus \Omega_C,E}
(\boldsymbol{u}_h^{H},\boldsymbol{\phi}_h) =
\int_{\Gamma_C} P(\boldsymbol{x}) \boldsymbol{n}(\boldsymbol{x}) \cdot \boldsymbol{\phi}_h d S(\boldsymbol{x}).
\end{aligned}
\right.\]

We will prove that the adjusted immersed boundary 
method is a perturbation of the 'hole' approach:
\begin{proposition}\label{Prop_hole_consis}
Let $\boldsymbol{u}_h^H$ and $\boldsymbol{u}_h^{\beta}$,
respectively, satisfy
Galerkin forms $(GF_H)$ and $(GF_{\beta})$, then there
is a $C > 0$ such that
$|| \boldsymbol{u}_h^H - \boldsymbol{u}_h^{\beta} ||_{\boldsymbol{H^1}(\Omega \setminus \Omega_C)} \leqslant C \sqrt{\beta}\|\boldsymbol{u}_h^\beta\|^{1/2}_{\boldsymbol{H^1}(\Omega_C)}$.
\end{proposition}

\begin{proof}
First we note that, as in the spirit of Theorem \ref{Th_solexist}, we consider Galerkin solutions in a subset of $\boldsymbol{H^1}$ whereas the solution to the 'continuous' weak formulation is not in $\boldsymbol{H^1}$. Formally $(GF_H)$ and $(GF_\beta)$ hold for test functions $\boldsymbol{\phi}_h$ from different sets, namely $\boldsymbol{V}_h(\Omega)$ and $\boldsymbol{V}_h(\Omega \setminus \Omega_C)$. If we choose $\boldsymbol{V}_h(\Omega_C)$ to correspond to Lagrangian basis functions associated to internal nodes in $\Omega_C$, then these basis functions vanish at $\Gamma_C$. Furthermore, within the set of Lagrangian basis functions that are associated with $\Omega \setminus \Omega_C$, there are Lagrangian basis functions associated with $\Gamma_C$, which have a compact, hence limited, support over $\Omega_C$ and in $\Omega \setminus \Omega_C$, then let $\boldsymbol{v} = \boldsymbol{u}_h^{\beta} - 
\boldsymbol{u}_h^H$, then subtraction of problems
$(GF_H)$ and $(GF_{\beta})$ gives
$$
a_{\Omega \setminus \Omega_C,E}(\boldsymbol{v},\boldsymbol{\phi}_h) = 
-\beta a_{\Omega_C,E}(\boldsymbol{u}_h^{\beta},\boldsymbol{\phi}_h).
$$
The left-hand side is a bounded and coercive form on
which we can apply Korn's Inequality. Furthermore,
boundedness of the right-hand side in 
$\boldsymbol{V}_h(\Omega \setminus \Omega_C)$ follows by
application of the Cauchy-Schwartz Inequality, hence
there is an $L > 0$ such that 
$|a_{\Omega_C,E}(\boldsymbol{u}_h^{\beta}, \boldsymbol{\phi}_h)| \leqslant L\|\boldsymbol{u}_h^\beta\|_{\boldsymbol{H^1}(\Omega_C)} \|\boldsymbol{\phi}_h\|_{\boldsymbol{H^1}(\Omega_C)}$. Herewith, we arrive at
$$
-\beta L \|\boldsymbol{u}_h^\beta\|_{\boldsymbol{H^1}(\Omega_C)}\|\boldsymbol{\phi}_h\|_{\boldsymbol{H^1}(\Omega_C)}\leqslant 
a_{\Omega \setminus \Omega_C,E}(\boldsymbol{v},\boldsymbol{\phi}_h)
\leqslant \beta L \|\boldsymbol{u}_h^\beta\|_{\boldsymbol{H^1}(\Omega_C)}\|\boldsymbol{\phi}_h\|_{\boldsymbol{H^1}(\Omega_C)}, \text{ for all } \boldsymbol{\phi}_h
\in \boldsymbol{V}_h(\Omega \setminus \Omega_C).
$$
Note that the $a_{\Omega \setminus \Omega_C}(\boldsymbol{v}, \boldsymbol{\phi}_h)$ contains $\boldsymbol{v}$ and $\boldsymbol{\phi}_h$ in $\Omega \setminus \Omega_C$, whereas the right-hand side of the inequality contains norms over $\Omega_C$. Using Korn's Inequality, and upon setting $\boldsymbol{\phi}_h = \boldsymbol{v}$ in $\Omega \setminus \Omega_C$, we arrive at 
\begin{align*}
&K || \boldsymbol{v} ||^2_{\boldsymbol{H^1}({\Omega \setminus \Omega_C})} \leqslant a_{\Omega \setminus \Omega_C,E}(\boldsymbol{v},\boldsymbol{v}) \leqslant
\beta L \|\boldsymbol{u}_h^\beta\|_{\boldsymbol{H^1}(\Omega_C)}\|\boldsymbol{\phi}_h\|_{\boldsymbol{H^1}(\Omega_C)}\\
&\Rightarrow || \boldsymbol{v} ||_{\boldsymbol{H^1}({\Omega \setminus \Omega_C})} \leqslant C \sqrt{\beta}\|\boldsymbol{u}_h^\beta\|^{1/2}_{\boldsymbol{H^1}(\Omega_C)},
\text{ where } C = \sqrt{\frac{L}{K}}\|\boldsymbol{\phi}_h\|^{1/2}_{\boldsymbol{H^1}(\Omega_C)}.
\end{align*}
\end{proof}

For the case of a spring-force boundary condition on $\partial \Omega$ one can derive
a compatibility condition. To this extent, we consider the following boundary value
problems, for the 'hole' problem:
\begin{equation*}
\label{Eq_elas_hole_prime}
(BVP_H')\left\{
\begin{aligned}
-\nabla\cdot\boldsymbol{\sigma}&=0,\qquad&\mbox{in $\Omega\backslash\overline{\Omega}_C$,}\\
\boldsymbol{\sigma}\cdot\boldsymbol{n}&=P(\boldsymbol{x})\boldsymbol{n}(\boldsymbol{x}),\qquad&\mbox{on $\partial\Omega_C$,}\\
\boldsymbol{\sigma}\cdot\boldsymbol{n}+\kappa\boldsymbol{u}&=\boldsymbol{0},\qquad&\mbox{on $\partial\Omega$,}
\end{aligned}
\right.
\end{equation*}
and for the immersed boundary problem:
\begin{equation*}
\label{Eq_elas_immerse_prime}
(BVP_I')\left\{
\begin{aligned}
-\nabla\cdot\boldsymbol{\sigma}&= \int_{\Gamma_C} P(\boldsymbol{x}') \boldsymbol{n}(\boldsymbol{x}') \delta(\boldsymbol{x} - \boldsymbol{x}') d S(\boldsymbol{x}),\qquad&\mbox{in $\Omega$,}\\
\boldsymbol{\sigma}\cdot\boldsymbol{n}+\kappa\boldsymbol{u}&=\boldsymbol{0},\qquad&\mbox{on $\partial\Omega$,}
\end{aligned}
\right.
\end{equation*}
Next we give a proposition regarding compatibility for the 'hole' approach and the
immersed boundary method for the case of a spring boundary condition:
\begin{proposition}
Let $\boldsymbol{u}_H$ and $\boldsymbol{u}_I$, respectively, be solutions to the 'hole' approach, see $(BVP_H')$ and to the immersed boundary approach, see $(BVP_I')$. Let $\Gamma_C$ denote the boundary of the cell, over which internal forces are exerted, and let $\partial \Omega$ be the outer boundary of $\Omega$. Then $$\int_{\partial \Omega} \kappa \boldsymbol{u}_H dS = \int_{\partial \Omega} \kappa \boldsymbol{u}_I d S = \int_{\Gamma_C} P(\boldsymbol{x}) \boldsymbol{n}(\boldsymbol{x}) d S.$$
\end{proposition}
\noindent
\begin{proof}
To prove that the above equation holds true, we integrate the PDE of both approaches over the computational domain $\Omega$.

For the immersed boundary approach, we get 
$$-\int_{\Omega}\nabla\cdot\boldsymbol{\sigma}d\Omega=\int_{\Omega}\sum_{j=1}^{N_S}P(\boldsymbol{x}_j)\boldsymbol{n}(\boldsymbol{x}_j)\delta(\boldsymbol{x}-\boldsymbol{x}_j)\Delta S(\boldsymbol{x}_j)d\Omega,$$
then after applying Gauss Theorem in the LHS and simplifying the RHS, we obtain
$$-\int_{\partial\Omega}\boldsymbol{\sigma}\cdot\boldsymbol{n}(\boldsymbol{x})dS=\sum_{j=1}^{N_S}P(\boldsymbol{x}_j)\boldsymbol{n}(\boldsymbol{x}_j)\Delta S(\boldsymbol{x}_j).$$
By substituting the Robin's boundary condition and letting $N_S\rightarrow\infty$, i.e. $\Delta S(\boldsymbol{x}_j)\rightarrow0$, the equation becomes
\begin{equation}
\label{Eq_immersed_integral}
\int_{\partial\Omega}\kappa\boldsymbol{u_I}dS=\int_{\Gamma_C}P(\boldsymbol{x})\boldsymbol{n}(\boldsymbol{x})dS.
\end{equation}

Subsequently, we do the same thing for the 'hole' approach. Then, we get
$$-\int_{\Omega}\nabla\cdot\boldsymbol{\sigma}d\Omega=0,$$
and we apply Gauss Theorem:
$$-\int_{\partial\Omega\cup\Gamma_C}\boldsymbol{\sigma}\cdot\boldsymbol{n}(\boldsymbol{x})dS=0,$$
which implies
$$-\int_{\partial\Omega}\boldsymbol{\sigma}\cdot\boldsymbol{n}(\boldsymbol{x})dS-\int_{\Gamma_C}\boldsymbol{\sigma}\cdot\boldsymbol{n}(\boldsymbol{x})dS=0.$$
Using the boundary conditions, we get
$$\int_{\partial\Omega}\kappa\boldsymbol{u}_HdS=\int_{\Gamma_C}P(\boldsymbol{x})\boldsymbol{n}(\boldsymbol{x})dS,$$
which is exactly the same as Eq (\ref{Eq_immersed_integral}). Hence we proved that
$$\int_{\partial \Omega} \kappa \boldsymbol{u}_H d S = \int_{\partial \Omega} \kappa \boldsymbol{u}_I dS = \int_{\Gamma_C} P(\boldsymbol{x}) \boldsymbol{n}(\boldsymbol{x}) d S.$$ 
\end{proof}

Hence, the two different approaches are consistent in the sense of global conservation of momentum and therefore the results from both approaches should be comparable. 
\subsection{The Smoothed Particle Approach}
\noindent
The Gaussian distribution is used here as an approximation for the Dirac Delta distribution. Hereby, we show that in the $n$-dimensional case, the Gaussian distribution is a proper approximation for the Dirac Delta distribution.  

\begin{lemma}\label{lemma_gaussian}
For an open domain $\Omega=(x_{1,1},x_{1,2})\times(x_{2,1},x_{2,2})\times\cdots\times(x_{n,1},x_{n,2})\subset\R^n, n\geqslant2$, let $$\delta_{\varepsilon}(\boldsymbol{x}-\boldsymbol{x'})=\FS{1}{(2\pi\varepsilon^2)^{n/2}}\exp\{-\FS{\|\boldsymbol{x}-\boldsymbol{x'}\|^2}{2\varepsilon^2}\},$$ where $\boldsymbol{x'}=(x'_1,\dots,x'_n)\in\Omega$, then\\
\vspace{0.2cm}
(i) $\lim_{\varepsilon\rightarrow 0^+}\delta_{\varepsilon}(\boldsymbol{x}-\boldsymbol{x'})\rightarrow0,$ for all $\boldsymbol{x}\neq\boldsymbol{x'}$;\\
\vspace{0.2cm}
(ii) Let $f(\boldsymbol{x})\in\mathbb{C}^2(\R^d)$ and $\|f(\boldsymbol{x})\|\leqslant M<+\infty$, then
there is a $C > 0$ such that 
$$ |\int_\Omega\delta_{\varepsilon}(\boldsymbol{x}-\boldsymbol{x'})f(\boldsymbol{x})d\Omega - f(\boldsymbol{x'})| \leqslant C \varepsilon^2 
\text{ as } \varepsilon \rightarrow 0^+.$$

\end{lemma}
\begin{proof}
(i) Since $\boldsymbol{x}\neq\boldsymbol{x'}$, $\lim_{\varepsilon\rightarrow0^+}\exp\{-\FS{\|\boldsymbol{x}-\boldsymbol{x'}\|^2}{2\varepsilon^2}\}\rightarrow0$. Thus, $$\lim_{\varepsilon\rightarrow 0^+}\delta_{\varepsilon}(\boldsymbol{x}-\boldsymbol{x'})\rightarrow0, \mbox{ for all $\boldsymbol{x}\neq\boldsymbol{x'}$.}$$
(ii) Now we consider
\begin{equation*}
\int_\Omega\delta_{\varepsilon}(\boldsymbol{x}-\boldsymbol{x'})f(\boldsymbol{x})d\Omega=\int_\Omega\FS{1}{(2\pi\varepsilon^2)^{n/2}}\exp\{-\FS{\|\boldsymbol{x}-\boldsymbol{x'}\|^2}{2\varepsilon^2}\}f(\boldsymbol{x})d\Omega.
\end{equation*}
Firstly, we integrate over the infinite domain:
\begin{align*}
&\int_{\mathbb{R}^n}\delta_{\varepsilon}(\boldsymbol{x}-\boldsymbol{x'})f(\boldsymbol{x})d\Omega\\
&=\frac{1}{(2\pi\varepsilon^2)^{n/2}}\int_{-\infty}^{+\infty}\cdots\int_{-\infty}^{+\infty}\exp\{-\FS{\|\boldsymbol{x}-\boldsymbol{x'}\|^2}{2\varepsilon^2}\}f(\boldsymbol{x})dx_n\cdots dx_1\\
&=\frac{1}{(2\pi\varepsilon^2)^{n/2}}\int_{-\infty}^{+\infty}\exp\{-\FS{(x_1-x'_1)^2}{2\varepsilon^2}\}\cdots\int_{-\infty}^{+\infty}\exp\{-\FS{(x_n-x'_n)^2}{2\varepsilon^2}\}\\
&f(\boldsymbol{x})dx_n\cdots dx_1.
\end{align*}

Again let $s_i=\frac{(x_i-x_i')-\frac{x_{i,1}+x_{i,2}}{2}}{\sqrt{2}\varepsilon},$ and furthermore $\xi_i=s_i+\frac{x_{i,1}+x_{i,2}}{2},i=\{1,2,\dots,n\}$. We denote $\boldsymbol{x_1}=(x_{1,1},x_{2,1},\dots,x_{n,1})$, $\boldsymbol{x_2}=(x_{1,2},x_{2,2},\dots,x_{n,2})$ and $\boldsymbol{x'}=(x'_1,x'_2\dots,x'_n)$.
By Taylor Expansion, $f(\boldsymbol{x})$ can be rewritten as 
\begin{align*}
&f(\boldsymbol{x})=f(\sqrt{2}\varepsilon\boldsymbol{s}+\frac{\boldsymbol{x_1}+\boldsymbol{x_2}}{2}+\boldsymbol{x'})\\
&=f(\boldsymbol{x'})+\nabla f(\boldsymbol{x'})(\sqrt{2}\varepsilon\boldsymbol{s}+\frac{\boldsymbol{x_1}+\boldsymbol{x_2}}{2})\\
&+\frac{1}{2!}(\sqrt{2}\varepsilon\boldsymbol{s}+\frac{\boldsymbol{x_1}+\boldsymbol{x_2}}{2})^T\boldsymbol{H}(\boldsymbol{x'})(\sqrt{2}\varepsilon\boldsymbol{s}+\frac{\boldsymbol{x_1}+\boldsymbol{x_2}}{2})+\mathcal{O}(\varepsilon^3)\\
&=f(\boldsymbol{x'})+\nabla f(\boldsymbol{x'})\sqrt{2}\varepsilon(\boldsymbol{s}+\frac{\boldsymbol{x_1}+\boldsymbol{x_2}}{2\sqrt{2}\varepsilon})\\
&+\varepsilon^2(\boldsymbol{s}+\frac{\boldsymbol{x_1}+\boldsymbol{x_2}}{2\sqrt{2}\varepsilon})^T\boldsymbol{H}(\boldsymbol{x'})(\sqrt{2}\varepsilon\boldsymbol{s}+\frac{\boldsymbol{x_1}+\boldsymbol{x_2}}{2\sqrt{2}\varepsilon})+\mathcal{O}(\varepsilon^3)\\
&=f(\boldsymbol{x'})+\nabla f(\boldsymbol{x'})\sqrt{2}\varepsilon\boldsymbol{\xi}+\varepsilon^2\boldsymbol{\xi}^T\boldsymbol{H}(\boldsymbol{x'})\boldsymbol{\xi}+\mathcal{O}(\varepsilon^3)
\end{align*}
where $\boldsymbol{H}(\boldsymbol{x'})$ is Hessian matrix of $f(\boldsymbol{x})$. For any non-negative integer $d$, 
\begin{equation*}
\int_{-\infty}^{+\infty}z^de^{-z^2}dz=
\left\lbrace
\begin{aligned}
&0, &\mbox{if $d$ is odd,}\\
&\Gamma(\FS{d+1}{2}), &\mbox{if $d$ is even.}
\end{aligned}
\right.
\end{equation*}

First we calculate 
\begin{align*}
&\int_{\mathbb{R}^n}\delta_{\varepsilon}(\boldsymbol{x}-\boldsymbol{x'})f(\boldsymbol{x})d\Omega\\
&=\frac{1}{(2\pi\varepsilon^2)^{n/2}}\int_{-\infty}^{+\infty}\exp\{-\FS{(x_1-x'_1)^2}{2\varepsilon^2}\}\cdots\int_{-\infty}^{+\infty}\exp\{-\FS{(x_n-x'_n)^2}{2\varepsilon^2}\}\\
&f(\boldsymbol{x})dx_n\cdots dx_1\\
&=\frac{1}{\pi^{n/2}}\int_{-\infty}^{+\infty}\exp\{(-s_1+\frac{x_{1,1}+x_{1,2}}{2})^2\}\cdots\int_{-\infty}^{+\infty}\exp\{(-s_n+\frac{x_{n,1}+x_{n,2}}{2})^2\}\\
&f(\sqrt{2}\varepsilon\boldsymbol{s}+\frac{\boldsymbol{x_1}+\boldsymbol{x_2}}{2}+\boldsymbol{x'})ds_n\cdots ds_1\\
&=\frac{1}{\pi^{n/2}}\int_{-\infty}^{+\infty}e^{-\xi_1^2}\cdots\int_{-\infty}^{+\infty}e^{-\xi_n^2}f(\sqrt{2}\varepsilon\boldsymbol{\xi}+\boldsymbol{x'})d\xi_n\cdots d\xi_1\\
&=\frac{1}{\pi^{n/2}}\int_{-\infty}^{+\infty}e^{-\xi_1^2}\cdots\int_{-\infty}^{+\infty}e^{-\xi_n^2}[f(\boldsymbol{x'})+\nabla f(\boldsymbol{x'})\sqrt{2}\varepsilon\boldsymbol{\xi}+\varepsilon^2\boldsymbol{\xi}^T\boldsymbol{H}(\boldsymbol{x'})\boldsymbol{\xi}\\
&+\mathcal{O}(\varepsilon^3)]d\xi_n\cdots d\xi_1\\
&=\frac{f(\boldsymbol{x'})}{\pi^{n/2}}\int_{-\infty}^{+\infty}e^{-\xi_1^2}\cdots\int_{-\infty}^{+\infty}e^{-\xi_n^2}d\xi_n\cdots d\xi_1\\
&+\frac{\sqrt{2}\varepsilon}{\pi^{n/2}}\int_{-\infty}^{+\infty}e^{-\xi_1^2}\xi_1f'_{x_1}(\boldsymbol{x'})\cdots\int_{-\infty}^{+\infty}e^{-\xi_n^2}\xi_nf'_{x_n}(\boldsymbol{x'})d\xi_n\cdots d\xi_1\\
&+\frac{\varepsilon^2}{\pi^{n/2}}\int_{-\infty}^{+\infty}e^{-\xi_1^2}(\sqrt{2}\xi_1+f''_{x_1,x_1}(\boldsymbol{x'})\xi_1^2+\sum_{i=1,i\neq1}^{n}f''_{x_1,x_i}(\boldsymbol{x'})\xi_1\xi_i\cdots\\
&\int_{-\infty}^{+\infty}e^{-\xi_n^2}(\sqrt{2}\xi_1+(f''_{x_n,x_n})(\boldsymbol{x'})\xi_n^2+\sum_{i=1,i\neq n}^{n}f''_{x_n,x_i}(\boldsymbol{x'})\xi_n\xi_id\xi_n\cdots d\xi_1+\mathcal{O}(\varepsilon^3)\\
&=f(\boldsymbol{x'})+\frac{\varepsilon^2}{\sqrt{\pi}}\Gamma(\frac{3}{2})\sum_{i=1}^{d}f''_{x_i,x_i}(\boldsymbol{x'})+\mathcal{O}(\varepsilon^3)\rightarrow f(\boldsymbol{x'}), \mbox{as $\varepsilon\rightarrow0^+$}.
\end{align*}

For the integral over the given domain $\Omega=(x_{1,1},x_{1,2})\times\cdots\times(x_{n,1},x_{n,2})$, it can be written as 

\begin{align*}
&\int_{x_{1,1}}^{x_{1,2}}\cdots\int_{x_{n,1}}^{x_{n,2}}dx_n\cdots dx_1\\
&=\int_{-\infty}^{+\infty}\cdots\int_{-\infty}^{+\infty}dx_n\cdots dx_1-\sum_{i=1}^{n}\int_{x_{1,1}}^{x_{1,2}}\cdots\int_{-\infty}^{x_{i,1}}\cdots\int_{x_{n,1}}^{x_{n,2}} dx_n\cdots dx_1\\
&-\sum_{i=1}^{n}\int_{x_{1,1}}^{x_{1,2}}\cdots\int_{x_{i,2}}^{+\infty}\cdots\int_{x_{n,1}}^{x_{n,2}} dx_n\cdots dx_1\\
&=(\sqrt{2}\varepsilon)^n\left[\int_{-\infty}^{+\infty}\cdots\int_{-\infty}^{+\infty}ds_n\cdots ds_1-\sum_{i=1}^{n}\int_{s_{1,1}}^{s_{1,2}}\cdots\int_{-\infty}^{s_{i,1}}\cdots\int_{s_{n,1}}^{s_{n,2}} ds_n\cdots ds_1\right.\\
&\left.-\sum_{i=1}^{n}\int_{\xi_{1,1}}^{\xi_{1,2}}\cdots\int_{\xi_{i,2}}^{+\infty}\cdots\int_{\xi_{n,1}}^{\xi_{n,2}} d\xi_n\cdots d\xi_1\right]\\
&=(\sqrt{2}\varepsilon)^n\left[\int_{-\infty}^{+\infty}\cdots\int_{-\infty}^{+\infty}d\xi_n\cdots d\xi_1-\sum_{i=1}^{n}\int_{\xi_{1,1}}^{\xi_{1,2}}\cdots\int_{-\infty}^{\xi_{i,1}}\cdots\int_{\xi_{n,1}}^{\xi_{n,2}} d\xi_n\cdots d\xi_1\right.\\
&\left.-\sum_{i=1}^{n}\int_{\xi_{1,1}}^{\xi_{1,2}}\cdots\int_{\xi_{i,2}}^{+\infty}\cdots\int_{\xi_{n,1}}^{\xi_{n,2}} d\xi_n\cdots d\xi_1\right],
\end{align*}
where $\xi_{i,1}=\frac{x_{i,1}-x'_i}{\sqrt{2}\varepsilon}$ and $\xi_{i,2}=\frac{x_{i,2}-x'_i}{\sqrt{2}\varepsilon}$.
Therefore, 
\begin{align*}
&\left|\int_\Omega\delta_{\varepsilon}(\boldsymbol{x}-\boldsymbol{x'})f(\boldsymbol{x})d\Omega-f(\boldsymbol{x'})\right|\\
&=\left|f(\boldsymbol{x'})+\frac{\varepsilon^2}{\sqrt{\pi}}\Gamma(\frac{3}{2})\sum_{i=1}^{d}f''_{x_i,x_i}(\boldsymbol{x'})+\mathcal{O}(\varepsilon^3)\right.\\
&-\frac{1}{\pi^{n/2}}\left[\sum_{i=1}^{n}\int_{\xi_{1,1}}^{\xi_{1,2}}e^{-\xi_1}\cdots\int_{-\infty}^{\xi_{i,1}}e^{-\xi_i}\cdots\int_{\xi_{n,1}}^{\xi_{n,2}}e^{-\xi_n}f(\sqrt{2}\varepsilon\boldsymbol{\xi}+\boldsymbol{x'}) d\xi_n\cdots d\xi_1\right.\\
&\left.\left.+\sum_{i=1}^{n}\int_{\xi_{1,1}}^{\xi_{1,2}}e^{-\xi_1}\cdots\int_{\xi_{i,2}}^{+\infty}e^{-\xi_i}\cdots\int_{\xi_{n,1}}^{\xi_{n,2}}e^{-\xi_n}f(\sqrt{2}\varepsilon\boldsymbol{\xi}+\boldsymbol{x'}) d\xi_n\cdots d\xi_1\right]-f(\boldsymbol{x'})\right|\\
&\leqslant \left|\frac{\varepsilon^2}{\sqrt{\pi}}\Gamma(\frac{3}{2})\sum_{i=1}^{d}f''_{x_i,x_i}(\boldsymbol{x'})+\mathcal{O}(\varepsilon^3)\right|\\
&+\frac{M}{2^{n-1}}\sum_{j=1}^n\prod_{i=1,i\neq j}^{n}[\erf(\xi_{j,2})-\erf(\xi_{j,1})+2][\erf(\xi_{i,2})-\erf(\xi_{i,1})]\\
&\leqslant \left|\frac{\varepsilon^2}{\sqrt{\pi}}\Gamma(\frac{3}{2})\sum_{i=1}^{d}f''_{x_i,x_i}(\boldsymbol{x'})+\mathcal{O}(\varepsilon^3)\right|\\
&+\frac{M}{2}\sum_{j=1}^n[\erf(\xi_{j,1})-\erf(\xi_{j,2})+2]\rightarrow0, \mbox{as $\varepsilon\rightarrow0^+$,}
\end{align*}
since $\|f(\boldsymbol{x})\|<M<+\infty$, $\xi_{i,1}\rightarrow -\infty$ and $\xi_{i,2}\rightarrow \infty$ respectively. Using
$1 - \erf(y) < \frac{2}{\sqrt{\pi}}~\exp(-y)$ for $y > 0$ and the fact
that $\exp(y) < \frac{1}{y^{\alpha}}$ as $y \rightarrow \infty$, we see
that the second term approximates zero faster than the first term.
Hence, we conclude that there is a $C>0$ such that
$$ |\int_\Omega\delta_{\varepsilon}(\boldsymbol{x}-\boldsymbol{x'})f(\boldsymbol{x})d\Omega - f(\boldsymbol{x'})| \leqslant C \varepsilon^2 
\text{ as } \varepsilon \rightarrow 0^+.$$	
\end{proof}

As a remark we add that setting $f(\boldsymbol{x}) = 1$, immediately shows
that there is a $C>0$ such that 
$$ |\int_\Omega\delta_{\varepsilon}(\boldsymbol{x}-\boldsymbol{x'})d\Omega - 1| \leqslant C \varepsilon^2 
\text{ as } \varepsilon \longrightarrow 0^+.$$
Using the result above, we start with analysing different approaches with only one relatively big cell in the computational domain. According to the model described in Eq (\ref{Eq_ElasTempeq}), the forces released on the boundary of the cell are the superposition of point forces on the midpoint of each surface element. For example, if we use a square shape to approximate the biological cell, then the forces are depicted in Figure \ref{Fig_mesh_2D}. Therefore, in $n$ dimensional case ($n > 1$), if the biological cell is a n-dimensional hypercube, then the forces can be rewritten as 
\begin{equation}
\label{Eq_temp}
\left.
\begin{aligned}
\boldsymbol{f}_t&=\sum_{j=1}^{N_S}P(\boldsymbol{x}_j)\boldsymbol{n}(\boldsymbol{x}_j)\delta(\boldsymbol{x}-\boldsymbol{x}_j)\Delta S(\boldsymbol{x}_j)\\
&=P\sum_{i=1}^n\{\boldsymbol{e_i}(\Delta x)^{n-1}[\delta(x_1-x'_1,\dots,x_i-(x'_i+\frac{\Delta x}{2}),\dots, x_n-x'_n)\\
&-\delta(x_1-x'_1,\dots,x_i-(x'_i-\frac{\Delta x}{2}),\dots, x_n-x'_n)]\},
\end{aligned}
\right.
\end{equation}
where $\boldsymbol{e_i}$ is the standard basis vector with $1$ in the i-th coordinate and $0's$ elsewhere, and $\Delta x$ is the length of cell boundary in each coordinate. For the smoothed force approach, we set $\delta(\boldsymbol{x})\approx\delta_{\varepsilon}(\boldsymbol{x})$. The force is given by 
\begin{equation}
\label{Eq_temp_smoothed}
\left.
\begin{aligned}
\boldsymbol{f}_\varepsilon&=P\sum_{i=1}^n\{\boldsymbol{e_i}(\Delta x)^{n-1}[\delta_\varepsilon(x_1-x'_1,\dots,x_i-(x'_i+\frac{\Delta x}{2}),\dots, x_n-x'_n)\\
&-\delta_\varepsilon(x_1-x'_1,\dots,x_i-(x'_i-\frac{\Delta x}{2}),\dots, x_n-x'_n)]\}.
\end{aligned}
\right.
\end{equation}

Following the same process in two dimensions \cite{peng2019pointforces} and thanks to the continuity of Gaussian distribution, as $\Delta x\rightarrow0$, the force converges to 
\begin{equation}
\label{Eq_temp_SPH}
\boldsymbol{f}_S=P(\Delta x)^n\nabla\delta_{\varepsilon}(\boldsymbol{x}-\boldsymbol{x'}).
\end{equation}

\begin{figure}[htpb]
\centering
\includegraphics[scale=0.35]{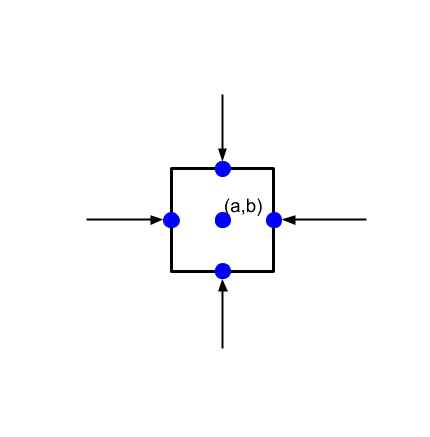}
\caption{We consider a rectangular shape cell in two dimensions, with the centre position at $(a,b)$. The forces exerted on the boundary are indicated by arrows}
\label{Fig_mesh_2D}
\end{figure}

\begin{theorem}\label{Th_delta_gauss}
Let $\boldsymbol{u}_h\subset\boldsymbol{V}_h(\Omega)$ be the Galerkin solution to the problem
\begin{equation}
\label{Eq_delta_2d}
(BVP)\left\{
\begin{aligned}
&\text{Find } \boldsymbol{u}_h \in \boldsymbol{V}_h(\Omega) \text{ such that }
a(\boldsymbol{u}_h,\boldsymbol{\phi}_h)=\int_{\Omega}\boldsymbol{f}_t \boldsymbol{\phi}_h d \Omega,\\
&\mbox{for all $\boldsymbol{\phi}_h \in \boldsymbol{V}_h(\Omega)$,}
\end{aligned}
\right.
\end{equation} 
and $\boldsymbol{u^\varepsilon}_h$ be the Galerkin solution to 
\begin{equation}
\label{Eq_smoothed_2d}
(BVP_\varepsilon)\left\lbrace 
\begin{aligned}
&\text{Find } \boldsymbol{u}^{\varepsilon}_h \in \boldsymbol{V}_h(\Omega) \text{ such that }
a(\boldsymbol{u}^{\varepsilon}_h,\boldsymbol{\phi}_h)=\int_{\Omega}\boldsymbol{f}_{\varepsilon} \boldsymbol{\phi}_h d \Omega,\\
&\mbox{for all $\boldsymbol{\phi}_h \in \boldsymbol{V}_h(\Omega)$.}
\end{aligned}
\right.	
\end{equation}
Then there is an $L_1>0$ such that
$||\boldsymbol{u^\varepsilon}_h-\boldsymbol{u}_h||_{\boldsymbol{H^1}(\Omega)} 
\leqslant L_1~ (\Delta x)^{(n-1)/2} ~\varepsilon$.     
\end{theorem}

\begin{proof}
Using bilinearity of $a(.,.)$ gives upon setting 
$\boldsymbol{w}=\boldsymbol{u}_h-\boldsymbol{u^\varepsilon}_h$ the
following equation:
$$
a(\boldsymbol{w},\boldsymbol{\phi}_h) = \int_{\Omega} (\boldsymbol{f}_t - \boldsymbol{f}_{\epsilon}) \cdot \boldsymbol{\phi}_h d \Omega.
$$
Using the result from Lemma \ref{lemma_gaussian} and the Triangle
Inequality, bearing in mind that $||\boldsymbol{e}_i|| = 1$ and that 
the basis field functions $\boldsymbol{\phi}_h$ are bounded, and after some algebraic manipulations, we can write the
right-hand side as
\begin{equation}
\begin{aligned}
|\int_{\Omega} (\boldsymbol{f}_t - \boldsymbol{f}_{\epsilon}) \cdot \boldsymbol{\phi}_h d \Omega| \leqslant C (\Delta x)^{n-1} \varepsilon^2.
\end{aligned}
\end{equation}
Coerciveness, see Lemma \ref{lemma_coerc}, and using 
$\boldsymbol{\phi}_h = \boldsymbol{w}$, gives 
$$
K ||\boldsymbol{w}||^2_{H_1(\Omega)} \leqslant a(\boldsymbol{w},\boldsymbol{w}) \leqslant 
C (\Delta x)^{n-1} \varepsilon^2,
$$
hence there is an $L_1>0$ such that
$||\boldsymbol{w}||_{H_1(\Omega)} \leqslant L~ (\Delta x)^{(n-1)/2}~ \varepsilon$, which immediately implies that 
$$||\boldsymbol{u}_h - \boldsymbol{u}_h^{\varepsilon}||_{H_1(\Omega)} \leqslant L_1~ (\Delta x)^{(n-1)/2}~ \varepsilon$$
\end{proof}

\begin{theorem}\label{Th_gauss_SPH}
Let $\boldsymbol{u^\varepsilon}_h$ be the solution to the boundary value problems
in Eq (\ref{Eq_smoothed_2d}),
and $\boldsymbol{u}^S_h$ the solution to 
\begin{equation}
\label{Eq_SPH_2d}
(BVP_{SP})\left\lbrace
\begin{aligned}
&\text{Find } \boldsymbol{u}^{\varepsilon}_h \in \boldsymbol{V}_h(\Omega) \text{ such that }
a(\boldsymbol{u}^{\varepsilon}_h,\boldsymbol{\phi}_h)=\int_{\Omega}\boldsymbol{f}_{S} \boldsymbol{\phi}_h d \Omega,\\
&\mbox{for all $\boldsymbol{\phi}_h \in \boldsymbol{V}_h(\Omega)$.}
\end{aligned} 
\right.	
\end{equation}
Then there is an $L_2>0$ such that 
$$
\frac{1}{(\Delta x)^n} || \boldsymbol{u}_h^S - \boldsymbol{u}_h^{\epsilon} ||_{\boldsymbol{H^1}(\Omega)} \leqslant L_2 \frac{(\Delta x)^2}{\varepsilon^3}.  
$$
\end{theorem}
\begin{proof}
Using bilinearity of $a(.,.)$ gives upon setting 
$\boldsymbol{w}=\boldsymbol{u}^{\varepsilon}_h-\boldsymbol{u}^S_h$ the
following equation:
$$
a(\boldsymbol{w},\boldsymbol{\phi}_h) = \int_{\Omega} (\boldsymbol{f}_{\varepsilon} - \boldsymbol{f}_{S}) \cdot \boldsymbol{\phi}_h d \Omega.
$$
Using Taylor's Theorem for multivariate functions on smoothed delta
distributions, we get the following result for the right-hand side:
\begin{equation}
\begin{aligned}
\int_{\Omega} (\boldsymbol{f}_{\epsilon} - \boldsymbol{f}_S) \cdot \boldsymbol{\phi}_h d \Omega =
\int_{\Omega} \frac{P(\boldsymbol{x}')}{48} (\Delta x)^{n+2}
\sum_{i=1}^{n} \boldsymbol{e}_i \frac{\partial^3 \delta_{\epsilon}(\boldsymbol{\hat{x}}-\boldsymbol{x}')}{\partial x_i^3} \cdot \boldsymbol{\phi}_h d \Omega,
\end{aligned}
\end{equation}
for $\boldsymbol{\hat{x}}$ between $\boldsymbol{x}$ and $\boldsymbol{x}'$.
The magnitude of the above expression can be estimated from above by
\begin{equation}
\begin{aligned}
|\int_{\Omega} (\boldsymbol{f}_{\epsilon} - \boldsymbol{f}_S) \cdot \boldsymbol{\phi}_h d \Omega| \leqslant
\frac{P(\boldsymbol{x}')}{48} (\Delta x)^{n+2} || \sum_{i=1}^{n} \boldsymbol{e}_i \frac{\partial^3 \delta_{\epsilon}}{\partial x_i^3} ||_{L^{\infty}(\Omega)}  || \boldsymbol{\phi}_h ||_{\boldsymbol{H^1}(\Omega)}.
\end{aligned}
\end{equation}
Using Lemma \ref{lemma_coerc}, this gives
$$
K || \boldsymbol{w} ||^2_{\boldsymbol{H^1}(\Omega)} \leqslant 
a(\boldsymbol{w},\boldsymbol{w}) \leqslant 
\frac{P(\boldsymbol{x}')}{48} (\Delta x)^{n+2} || \sum_{i=1}^{n} \boldsymbol{e}_i \frac{\partial^3 \delta_{\epsilon}}{\partial x_i^3} ||_{L^{\infty}(\Omega)}  || \boldsymbol{\phi}_h ||_{\boldsymbol{H^1}(\Omega)}.
$$
Division by $K || \boldsymbol{w} ||^2_{\boldsymbol{H^1}(\Omega)}$ gives
$$
|| \boldsymbol{w} ||_{\boldsymbol{H^1}(\Omega)} \leqslant 
\frac{P(\boldsymbol{x}')}{48} (\Delta x)^{n+2} || \sum_{i=1}^{n} \boldsymbol{e}_i \frac{\partial^3 \delta_{\epsilon}}{\partial x_i^3} ||_{L^{\infty}(\Omega)}.
$$
We bear in mind that $\frac{\partial^3 \delta_{\epsilon}}{\partial x_i^3} = \mathcal{O}(\varepsilon^{-3})$, 
this implies that there is an $L_2>0$ such that
$$
\frac{1}{(\Delta x)^n} || \boldsymbol{u}_h^S - \boldsymbol{u}_h^{\epsilon} ||_{\boldsymbol{H^1}(\Omega)} \leqslant L_2 \frac{(\Delta x)^2}{\varepsilon^3}.  
$$
\end{proof}

With the two theorems above, we have proved that the solution to $(BVP_\varepsilon)$ converges to the solution to $(BVP)$, and the solution to $(BVP_{SP})$ converges to the solution to $(BVP_\varepsilon)$. Hence, we can derive the following theorem:
\begin{theorem}
Let $\boldsymbol{u}_h$ be the Galerkin solution to $(BVP)$ and $\boldsymbol{u}_h^S$ be the solution to $(BVP_{SP})$, let
$\varepsilon = \mathcal{O}(\Delta x)^p$ and ${\Delta x}\rightarrow{0}$.
If $0 < p < (2+n)/3$ then
$\boldsymbol{u}_h^S$ converges to $\boldsymbol{u}_h$ in the $H^1$--norm,
and $\boldsymbol{u}_h^S$ converges to $\boldsymbol{u}$ in the 
$H^1$--norm. 
\end{theorem}
\begin{proof}
Denote 	$\boldsymbol{u}_h$ and $\boldsymbol{u}^S_h$ to be the Galerkin solution to $(BVP_\varepsilon)$ and $(BVP_{SP})$.  
Firstly, we consider 
\begin{align*}
\|\boldsymbol{u}_h-\boldsymbol{u}^S_h\|&=
\|\boldsymbol{u}_h-\boldsymbol{u^\varepsilon}_h+\boldsymbol{u^\varepsilon}_h-\boldsymbol{u}^S_h\|\leqslant \|\boldsymbol{u}_h-\boldsymbol{u^\varepsilon}_h\|+\|\boldsymbol{u^\varepsilon}_h-\boldsymbol{v^\varepsilon}_h\| \\ 
&\leqslant L_1 (\Delta x)^{(n-1)/2} \varepsilon + L_2 \frac{(\Delta x)^{2+n}}{\varepsilon^3}  \\ 
&= L_1 (\Delta x)^{(n-1)/2+p} + L_2 (\Delta x)^{2+n-3p} \rightarrow 0,\\
&\text{ as } \Delta x \rightarrow 0, \text{ if } 0 < p < (2+n)/3.
\end{align*}

From this inequality, we conclude that the finite element solution
of the smooth particle method converges to the solution of the 
immersed boundary method upon letting $\Delta x \rightarrow 0$ and
choosing $\varepsilon = \mathcal{O}(\Delta x)^p)$ for $0 < p < (2+p)/3$.
\end{proof}

\section{Numerical Results in Two Dimensions}
\noindent
To demonstrate the consistency between the immersed boundary approach and two alternative methods, we consider a square-shape cell in the computational domain. A homogeneous boundary condition is imposed for the exterior boundary of the computational domain. The parameter values are listed in Table \ref{Tbl_paras}. All of them are educated guesses in this study and they are dimensionless.

\begin{table}[htpb]\footnotesize
\centering 
\caption{Parameter values used in the comparison of plasticity and morphoelasticity}
\begin{tabular}{m{3cm}<{\centering}m{7cm}<{\centering}m{1cm}<{\centering}}
	\toprule
	{\bf Parameter}& {\bf Description} & {\bf Value} \\
	\toprule
	$E$ & Substrate stiffness & $1$ \\
	$\beta$ & Factor between the cell stiffness and the substrate stiffness in Eq \ref{Eq_stiffness} & $10^{-5}$\\
	$R$ & Length of side of square-shape biological cells & $6$\\
	$\nu$ & Poisson's ratio & $0.48$\\
	$P$ & Magnitude of the temporary forces per unit length& $1$\\
	$x_0$ & Length of the computational domain in x-coordinate & $20$\\
	$y_0$ & Length of the computational domain in y-coordinate & $20$\\
	$w_x$ & Length of the wound domain in x-coordinate & $10$\\
	$w_y$ & Length of the wound domain in y-coordinate & $10$\\
	\bottomrule
\end{tabular}
\label{Tbl_paras}
\end{table}

According to Proposition \ref{Prop_hole_consis}, to compare the immersed boundary approach and the 'hole' approach, the stiffness inside the biological cell needs to be adjusted, since two approaches are consistent with $\beta\rightarrow 0$. However, in the implementation, we can only select a very small positive value instead of $\beta=0$.

\begin{figure}[htpb]
\centering
\subfigure[Immersed boundary approach]{
	\includegraphics[width=0.31\textwidth]{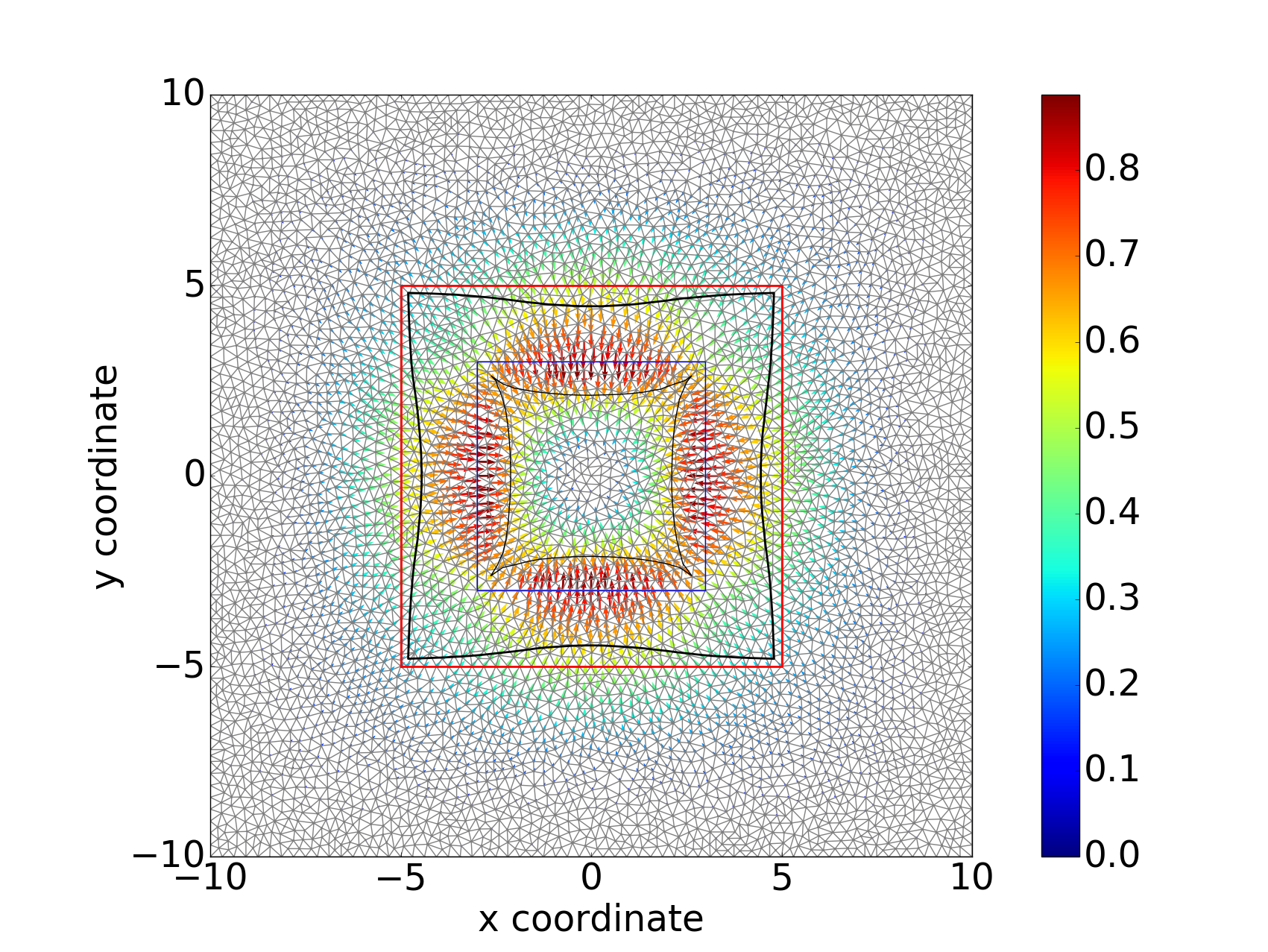}}
\subfigure[The 'hole' approach]{
	\includegraphics[width=0.31\textwidth]{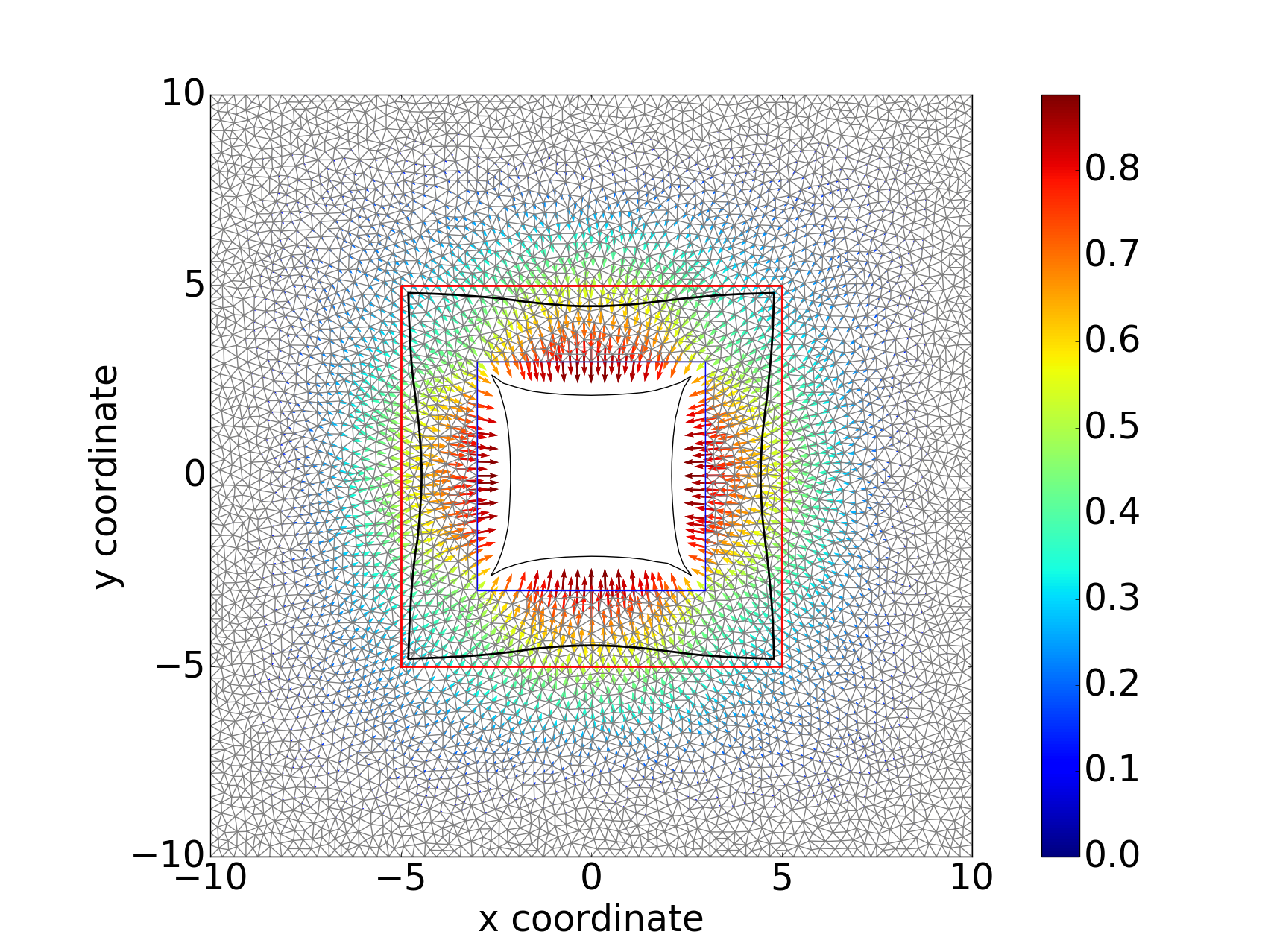}}
\subfigure[Smoothed particle approach]{
	\includegraphics[width=0.31\textwidth]{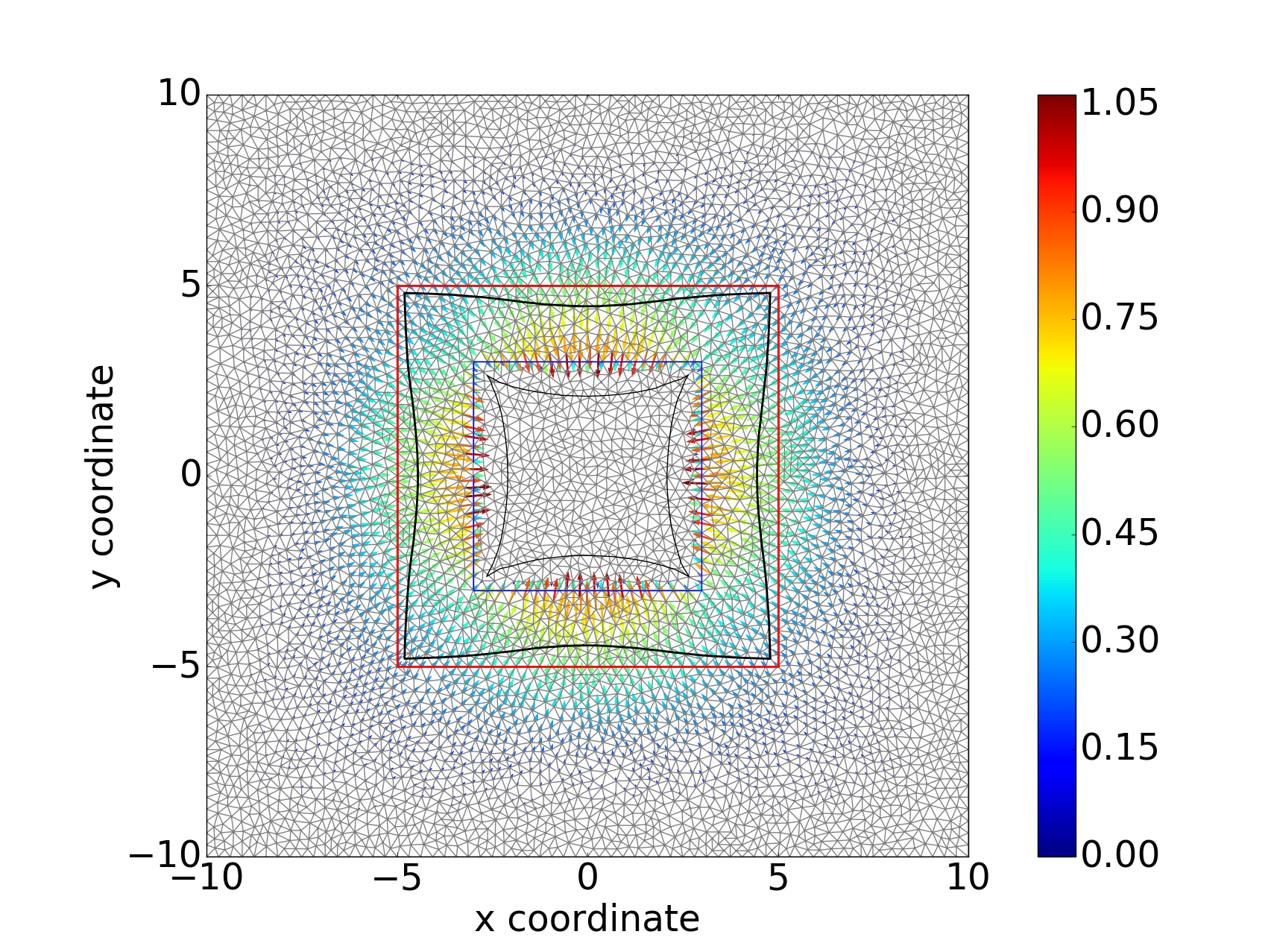}}
\caption{For the different stiffness inside and outside of the cell, the solution (i.e. the displacement) is showed in each approach. Black curves show the deformed region of vicinity and the cell, and blue curve represents the cell.}
\label{Fig_dis_mod}
\end{figure}

Numerical results are presented in Figure \ref{Fig_dis_mod}, Table \ref{Tbl_immersed_conti_area} and Table \ref{Tbl_immersed_conti_error}. From the figure, there is no significant difference, except that in the smoothed particle approach, the displacement is a little bit larger than in the other two approaches. The reduction ratio of either the vicinity region or the cell appears to yield a tiny difference, which implies that three approaches are numerically consistent. However, the 'hole' approach takes slightly more computation time than the other two approaches. Therefore and due to the numerical complications in needing adaptive meshes, it will not be elected when we deal with the displacement and deformation of large number of cells, even though its convergence rate improves significantly comparing to the immersed boundary approach. As for the smoothed particle approach, the convergence rate of the $L_2-$norm does not improve, while the computational efficiency does. 

\begin{table}[htpb]\footnotesize
\centering
\caption{The percentage of area change of cell and vicinity region, and time cost of various approaches, if the stiffness is different inside and outside the biological cell.}
\begin{tabular}{m{3.5cm}<{\centering}m{2.8cm}<{\centering}m{2.8cm}<{\centering}m{2.8cm}<{\centering}}
	\toprule
	& {\bf The immersed boundary approach} & {\bf The 'hole' approach} & {\bf The smoothed particle approach}  \\
	\midrule
	Cell Area Reduction Ratio(\%) & $45.84096$ & $45.71401$ & $45.18525$  \\
	Vicinity Area Reduction Ratio(\%) & $14.274671$ & $14.15804$ & $14.16180$ \\
	Time Cost$(s)$ & $0.78643$ & $1.10195$ & $0.75832$\\
	\bottomrule
\end{tabular}
\label{Tbl_immersed_conti_area}
\end{table}

\begin{table}[htpb]\footnotesize
\centering
\caption{The $L^2-norm $of the solution (i.e. the displacement) with different mesh size in each approach, if the stiffness is different inside and outside the biological cell.}
\begin{tabular}{m{2cm}<{\centering}m{3.5cm}<{\centering}m{3.5cm}<{\centering}m{3.5cm}<{\centering}}
	\toprule
	& {\bf The immersed boundary approach} & {\bf The 'hole' approach} & {\bf The smoothed particle approach}  \\
	\midrule
	h & $5.8833092$ & $5.9256424$ & $5.8981846$  \\
	h/2 & $5.9302898$ & $5.952170$ & $5.9324678$ \\
	h/4 & $5.9484929$ & $5.9593735$ & $5.9486686$\\
	Convergence rate & $1.36788$ & $1.88060$ & $1.08102$\\
	\bottomrule
\end{tabular}
\label{Tbl_immersed_conti_error}
\end{table}
\section{Conclusion}
\label{conclusion}
For the dimensionality exceeding one, the existence of Dirac Delta
distributions in the elasticity equation results into a singular
solution. We analyse the solutions based on Galerkin approximations
with Lagrangian basis functions for different approaches that are
consistent if cell sizes and smoothness parameters tend to zero.
We have shown that all the alternative approaches are numerically consistent
with the immersed boundary approach. The current paper has
investigated and extended earlier findings to multi-dimensionality. The current analysis has been 
carried out for simple, linear elasticity. In the future, 
we plan to extend our findings to the visco-elasticity 
equations. This visco-elastic model contains a damping 
term, and still retains a linear nature. 
Furthermore, we are also interested in analyzing the
above considered principles for a morphoelastic model.
A morpho-elastic model has the major advantage of incorporating
permanent deformations. A major complication is its 
nonlinear nature.

\section*{Acknowledgment}
\noindent
Authors acknowledge the China Scholarship Council (CSC) for financial support to this project.

\bibliography{template}


%
%
%
\end{document}